\documentclass[a4paper, 11pt, twoside]{article}

\pdfoutput=1

\usepackage{comment}
\usepackage{cite}
\usepackage[utf8]{inputenc}
\usepackage{enumerate}
\usepackage[ngerman, english]{babel}
\usepackage{array}
\usepackage{amsmath}
\usepackage{amssymb}
\usepackage{amsthm}
\usepackage{mathtools}
\usepackage{graphicx} 
\usepackage{setspace}
\usepackage{color}
\usepackage{scrlayer-scrpage}
\usepackage{titlesec}
\usepackage{tikz}
\usepackage[pdftex]{hyperref}
\usepackage[a4paper, lmargin=3.5cm, rmargin=3.0cm, top=3.0cm, bottom=2.5cm, head=14.5pt]{geometry}

 \hypersetup{
	pdfauthor = {Mareike Wolff},
	colorlinks = true,  
	linkcolor = black,
	citecolor = black,
	}		
  
\newcommand\dC{\mathbb{C}}

\newcommand\dN{\mathbb{N}}

\newcommand\dR{\mathbb{R}}
\newcommand\dZ{\mathbb{Z}}

\newcommand\cA{\mathcal{A}}
\newcommand\cB{\mathcal{B}}
\newcommand\cC{\mathcal{C}}
\newcommand\cD{\mathcal{D}}
\newcommand\cE{\mathcal{E}}
\newcommand\cF{\mathcal{F}}
\newcommand\cG{\mathcal{G}}
\newcommand\cH{\mathcal{H}}

\newcommand\cJ{\mathcal{J}}
\newcommand\cK{\mathcal{K}}

\newcommand\cN{\mathcal{N}}
\newcommand\cO{\mathcal{O}}
\newcommand\cP{\mathcal{P}}
\newcommand\cQ{\mathcal{Q}}
\newcommand\cR{\mathcal{R}}
\newcommand\cS{\mathcal{S}}
\newcommand\cT{\mathcal{T}}
\newcommand\cU{\mathcal{U}}
\newcommand\cV{\mathcal{V}}

\newcommand\epsi{\varepsilon}
\newcommand\vi{\varphi}
\newcommand\teta{\vartheta}

\DeclareMathOperator{\re}{Re}
\DeclareMathOperator{\im}{Im}
\DeclareMathOperator{\dens}{dens}
\DeclareMathOperator{\meas}{meas}
\DeclareMathOperator{\dist}{dist}
\DeclareMathOperator{\diam}{diam}
\DeclareMathOperator{\sgn}{sgn}
\DeclareMathOperator{\sing}{sing}
\DeclareMathOperator{\length}{length}

\newcommand\diff{\text{d}}

\newtheorem{thm}{Theorem}[section]
\newtheorem*{importthm}{Theorem}
\newtheorem{lemma}[thm]{Lemma}
\newtheorem{cor}[thm]{Corollary}

\theoremstyle{definition}

\newtheorem{remark}[thm]{Remark}

\numberwithin{equation}{section}

\cohead{\small Julia sets of measure zero}
\cehead{\small Mareike Wolff}

\mathtoolsset{showonlyrefs=true}

\begin{document}
\title{A class of Newton maps with Julia sets of Lebesgue measure zero}
\author{Mareike Wolff}
\date{}
\maketitle

\begin{abstract}
\noindent Let $g(z)=\int_0^zp(t)\exp(q(t))\,dt+c$ where $p,q$ are polynomials and $c\in\dC$, and let $f$ be the function from Newton's method for $g$. We show that under suitable assumptions the  Julia set of $f$ has Lebesgue measure zero. Together with a theorem by Bergweiler, our result implies that $f^n(z)$ converges to zeros of $g$ almost everywhere in $\dC$ if this is the case for each zero of $g''$. In order to prove our result, we establish general conditions ensuring that Julia sets have Lebesgue measure zero.\\

\end{abstract}


\section{Introduction and results}   \label{sec_intro}
For a rational function $f:\hat\dC\to\hat\dC$ that is not constant and not a Möbius transformation, or a transcendental meromorphic function $f:\dC\to\hat\dC$, let $f^n$ denote the $n$th iterate of $f$. The \textit{Fatou set}, $\cF(f)$, is the set of all $z$ such that all iterates $f^n$ are defined and form a normal family in a neighbourhood of $z$. Its complement, $\cJ(f)$, is called the \textit{Julia set}. For an introduction to the iteration theory of meromorphic functions, see, for example, \cite{steinmetz93} for rational functions and \cite{bergweiler93s} for transcendental functions.

Let $g$ be a non-constant meromorphic function. Newton's root finding method for $g$ consists of iterating the function
\begin{equation}   \label{eq_f}
f(z)=z-\frac{g(z)}{g'(z)}.
\end{equation}
We also call $f$ the \textit{Newton map} corresponding to $g$. The zeros of $g$ are precisely the attracting fixed points of $f$, and the simple zeros of $g$ are even superattracting fixed points. Recall that, more generally, a periodic point $z_0$ of period $p$ of a meromorphic function $f$ is called \textit{attracting}, \textit{indifferent} or \textit{repelling} depending on whether $|(f^p)'(z_0)|<1$, $|(f^p)'(z_0)|=1$ or $|(f^p)'(z_0)|>1$. The periodic point $z_0$  is called \textit{superattracting} if $(f^p)'(z_0)=0$. An indifferent periodic point $z_0$ is called \textit{rationally indifferent} if $(f^p)'(z_0)$ is a $q$th root of unity for some $q\in\dN$, otherwise it is called \textit{irrationally indifferent}.

We will investigate the Lebesgue measure of Julia sets of Newton maps corresponding to functions of the form
\begin{equation}  \label{eq_g}
g(z)=\int_0^z p(t)e^{q(t)}\,dt+c
\end{equation}
where $p$ is a polynomial with $p\not\equiv0$, $q$ is a non-constant polynomial and $c\in\dC.$ 

In the following, we will assume that $g$ is not of the form
\begin{equation}  \label{eq_notg}
g(z)=\tilde p(z)e^{\tilde q(z)}
\end{equation}
with polynomials $\tilde p$ and $\tilde q$. Then $g$ has infinitely many zeros and $f$ is transcendental. Newton's method for functions of the form \eqref{eq_notg} has been studied by Haruta \cite{haruta99}. 

Let $\dist(\cdot,\cdot)$ denote the Euclidean distance in $\dC$. We will prove the following result.

\begin{thm}   \label{thm_main}
Let $g$ be of the form \eqref{eq_g} but not of the form \eqref{eq_notg}, and let $f$ be the corresponding Newton map. Denote the zeros of $g''$ which are not zeros of $g$ or $g'$ by $z_1,...,z_N$. Suppose that  for all $j\in\{1,...,N\}$, the point $z_j$ is attracted by a periodic cycle, that is, there exists a periodic cycle $\cC$ of $f$ such that $\lim_{n\to\infty}\dist(f^n(z_j),\,\cC)=0$.  Then the Lebesgue measure of $\cJ(f)$ is zero.
\end{thm}

Jankowski \cite[§3]{jankowski96} proved that if $f$ is the Newton map corresponding to a function $g$ of the form $g(z)=r(z)e^{az}+b$ where $r$ is a rational function and $a,b\in\dC\setminus\{0\}$, and if for each of the zeros, $z_1,...,z_N$, of $g''$ that are not zeros of $g$ or $g'$, the iterates $f^n(z_j)$ converge to a finite limit as $n\to\infty$, then the Julia set of $f$ has Lebesgue measure zero. Note that if $r$ is a polynomial, then $g$ can be written in the form \eqref{eq_g} with $q(t)=at$ and $p(t)=ar(t)+r'(t)$. Also, under the assumptions of Jankowski's result, $f^n(z_j)$ is attracted by a cycle of period $1$ for all $j\in\{1,...,N\}$. So Jankowski's theorem for polynomial $r$ is a special case of Theorem \ref{thm_main}. The essential new difficulties we have to deal with in our proof come from the fact that we allow $q$ to have degree greater than one.

Bergweiler \cite[Theorem 3]{bergweiler93} also investigated Newton's method for functions of the form \eqref{eq_g}. He proved the following result.

\begin{importthm}[Bergweiler]  
Let $g$ be of the form \eqref{eq_g} but not of the form $e^{az+b}$ with $a,b\in\dC$, and let $f$ be the corresponding Newton map. Denote the zeros of $g''$ which are not zeros of $g$ or $g'$ by $z_1,...,z_N$. If $f^n(z_j)$ converges to a finite limit for all $j\in\{1,...,N\}$, then $f^n(z)$ converges to zeros of $g$ on an open dense subset of $\dC$. 
\end{importthm}

It is not difficult to see that under the assumptions of Bergweiler's theorem, $f^n(z_j)$ converges to an attracting fixed point of $f$ and hence a zero of $g$ for all $j\in\{1,...,N\}$. So the Theorem says that $f^n(z)$ converges to zeros of $g$ on an open dense subset of $\dC$, provided this is the case for each zero $z$ of $g''$ that is not a zero of $g$ or $g'$.

A component $\cU$ of the Fatou set  $\cF(f)$ is called \textit{periodic} if there is $p\in\dN$ with $f^p(\cU)\subset \cU$, the component $\cU$ is called \textit{preperiodic} if there is $l\in\dN$ such that $f^l(\cU)$ is contained in a periodic Fatou component, and $\cU$ is called a \textit{wandering domain} if it is not (pre)periodic. It is known (see, e.g., \cite[§4]{bergweiler93s}) that if $\cU$ is a periodic Fatou component of period $p$ of $f$, then either $f^{np}|_\cU$ converges to an attracting periodic point in $\cU$ (\textit{immediate basin of attraction}),  $f^{np}|_\cU$ converges to a rationally indifferent periodic point in $\partial\cU$ (\textit{parabolic domain}),  $f^{np}|_\cU$ converges to some $z_0\in\partial\cU$ and $f^p(z_0)$ is not defined (\textit{Baker domain}), or $f^{np}|_\cU$ is conjugate to a rotation of a disk (\textit{Siegel disk}) or an annulus (\textit{Herman ring}).

Bergweiler's theorem is proved by showing that under the given assumptions, $f$ has neither wandering domains nor parabolic domains, Baker domains, Siegel disks or Herman rings.

The following corollary is a direct consequence of Theorem \ref{thm_main} and Bergweiler's theorem.

\begin{cor}  \label{cor_main}
Let $g$ be of the form \eqref{eq_g} but not of the form \eqref{eq_notg}, and let $f$ be the corresponding Newton map. Denote the zeros of $g''$ which are not zeros of $g$ or $g'$ by $z_1,...,z_N$. If $f^n(z_j)$ converges to a finite limit for all $j\in\{1,...,N\}$, then $f^n(z)$ converges to zeros of $g$ for almost all $z\in\dC$.
\end{cor}

For example, the assumptions of Corollary \ref{cor_main} and hence those of Theorem \ref{thm_main} are satisfied for $g(z)=\int_0^ze^{-t^2}\,dt+c$ with $-\sqrt{\pi}/2<c<\sqrt{\pi}/2$, see \cite[§8]{bergweiler93}. Clearly, the conclusion of Corollary \ref{cor_main} cannot be true if there exists a cycle of period at least two in $\cF(f)$. 

In order to prove Theorem \ref{thm_main}, we will first prove a general Theorem giving conditions ensuring that the Julia set of a meromorphic function has Lebesgue measure zero. This may be of independent interest. For a meromorphic function $f$, we denote by $\sing(f^{-1})$ the set of singular values of $f$, that is, the set of critical and asymptotic values of $f$ and limit points of those. For $n\ge0$, let $\cN_n=\{z:\,f^n(z)\text{ is not defined}\}$. Let
$$\cP(f):=\overline{\bigcup_{n=0}^\infty f^n(\sing(f^{-1})\setminus\cN_n)}$$
denote the \textit{postsingular set} of $f$. 
For $z_0\in\dC$ and $r>0$, let $\cD(z_0,r)$ denote the open disk centred at $z_0$ with radius $r$. Also, let $\meas(\cdot)$ denote Lebesgue measure, and for measurable $\cA,\cB\subset\dC$  with $0<\meas(\cB)<\infty$, let 
$$\dens(\cA,\cB)=\frac{\meas(\cA\cap\cB)}{\meas(\cB)}$$
 denote the density of $\cA$ in $\cB$. 
 
 Following \cite{mcmullen87}, we call a measurable set $\cA\subset\dC$ \textit{thin at} $\infty$ if there exist $R_0,\epsi_0>0$ such that for all $z\in\dC$, we have 
 $$\dens(\cA, \cD(z,R_0))<1-\epsi_0.$$
Additionally, we introduce the concept that $\cA$ is \textit{thin at} $z_0\in\dC$ if there exist $\delta_1, \epsi_1>0$ such that for all $z\in\overline{\cD(z_0,\delta_1)}$, we have
 \begin{equation}  \label{eq_thinatz}
 \dens(\cA, \cD(z, |z-z_0|))<1-\epsi_1.
 \end{equation}
 We call $\cA$ \textit{uniformly thin at} $\cB\subset\dC$ if there are $\delta_1,\epsi_1>0$ such that \eqref{eq_thinatz} holds for all $z_0\in\cB$.
 
\begin{thm}   \label{thm_zero_measure}
Let $f$ be a meromorphic function that is not constant and not a Möbius transformation. Suppose that there exists $R_1>0$ such that
\begin{enumerate}[(i)]
\item $\cP(f)\cap\cJ(f)\cap\overline{\cD(0,R_1)}$ is a finite set;
\item $\cJ(f)$ is thin at $\infty$;
\item $\cJ(f)$ is uniformly thin at $(\cP(f)\cap\dC)\setminus\overline{\cD(0,R_1)}$.
\end{enumerate}
Then the Lebesgue measure of $\cJ(f)$ is zero.
\end{thm}

McMullen \cite[Proposition 7.3]{mcmullen87} proved that if $f$ is entire, $\cP(f)$ is compact, $\cP(f)\cap\cJ(f)=\emptyset$ and $\cJ(f)$ is thin at $\infty$, then $\meas(\cJ(f))=0$. A meromorphic function $f$ for which $\cP(f)$ is compact and does not intersect $\cJ(f)$ is called \textit{hyperbolic}. There are various results on iteration of hyperbolic meromorphic functions; see, for example, \cite{stallard99, rempe-sixsmith17, rippon-stallard99, bergweiler-fagella-rempe15, zheng15}. 
Stallard \cite{stallard90} extended McMullen's result to entire functions $f$ with possibly unbounded postsingular set such that $\dist(\cP(f),\cJ(f))>0$ and $\cJ(f)$ is thin at $\infty$. Meromorphic functions $f$ with $\dist(\cP(f),\cJ(f))>0$ are sometimes called $\textit{topologically hyperbolic}$, and have also been considered in \cite{mayer-urbanski10, bfjk20}. 

Jankowski \cite{jankowski96, jankowski97} extended Stallard's result by allowing that $f$ is meromorphic and that there are certain exceptions to the condition  $\dist(\cP(f),\cJ(f))>0$. A more general result was later obtained by Zheng \cite[Theorem 5]{zheng02} who proved that if $f$ is a meromorphic function such that the set $\cP(f)\cap\cJ(f)$ is finite, there exists $R>0$ such that $\dist((\cJ(f)\cap\dC)\setminus\cD(0,R), \cP(f))>0$ and $\cJ(f)$ is thin at $\infty$, then $\meas(\cJ(f))=0$. 

The results by McMullen, Stallard, Jankowki and Zheng mentioned above are special cases of Theorem \ref{thm_zero_measure} since the condition that $\dist((\cJ(f)\cap\dC)\setminus\cD(0,R),\cP(f))>0$ implies that $\cJ(f)$ is uniformly thin at $(\cP(f)\cap\dC)\setminus\overline{\cD(0,R')}$ for $R'>R$. Our theorem is the first of this kind to allow infinitely many postsingular values in the Julia set or an unbounded sequence of postsingular values whose distance to the Julia set tends to zero. 

In general, the condition that $\cJ(f)$ is thin at $\infty$ cannot be dropped. If $|\alpha|$ is small, then the postsingular set of $f(z)=\sin(\alpha z)$ is a compact subset of $\cF(f)$, and McMullen \cite{mcmullen87} showed that $\cJ(f)$ has positive measure. However, there are results where instead of assuming that $\cJ(f)$ is thin at $\infty$ other conditions are imposed, see \cite[Theorem 8]{eremenko-lyubich92}, \cite[Theorems 3 and 4]{zheng02}.

This article is structured as follows. In Section \ref{sec_zero_measure}, we prove Theorem \ref{thm_zero_measure}. In the remaining part of this paper, we prove Theorem \ref{thm_main}. First, in Section \ref{sec_changeOfVar}, we introduce the change of variables $w=q(z)$. In Section \ref{sec_asymptotics}, we give asymptotic representations of $g$ and $f$. In Section \ref{sec_partition}, we introduce a class of subsets of $\dC$ whose preimages under $q$ are connected to the asymptotic behaviour of $f$. In Section \ref{sec_singularities}, we investigate the postsingular set of $f$. In Section \ref{sec_q(F)1}-\ref{sec_q(F)_all}, we investigate the location and size of the set $q(\cF(f))$. Finally, in Section \ref{sec_proof}, we complete the proof of Theorem \ref{thm_main}.


\section{Julia sets of zero measure}  \label{sec_zero_measure}

In this section, we prove Theorem \ref{thm_zero_measure}. The following Lemma is an easy consequence of the well-known Koebe $1/4$-theorem and Koebe distortion theorem (see, e.g., \cite{pommerenke75}).

\begin{lemma}   \label{lemma_koebe}
Let $z_0\in\dC$ and $r>0$, and let $f:\cD(z_0,r)\to\dC$ be holomorphic and injective. Then
$$f(\cD(z_0,r))\supset\cD\left(f(z_0),\frac{1}{4}|f'(z_0)|r\right).$$
Moreover, for $\rho\in(0,1)$,
$$f(\cD(z_0,\rho r))\subset\cD\left(f(z_0),\frac{\rho}{(1-\rho)^2}|f'(z_0)|r\right)$$
and
$$\frac{\min_{z\in\cD(z_0,\rho r)}|f'(z)|}{\max_{z\in\cD(z_0,\rho r)}|f'(z)|}\ge\left(\frac{1-\rho}{1+\rho}\right)^4.$$
\end{lemma}

For $\cA\subset\dC$, denote the forward orbit of $\cA$ by
$$\cO^+(\cA):=\bigcup_{n=0}^\infty f^{n}(\cA\setminus\cN_n)$$
where $\cN_n=\{z:\,f^n(z) \text{ is not defined}\}$, and for $\cB\subset\hat\dC$, let 
$$\cO^-(\cB):=\bigcup_{n=1}^\infty f^{-n}(\cB)$$
be the backward orbit of $\cB$. For $z\in\hat\dC$, write
$$\cO^\pm(z):=\cO^\pm(\{z\}).$$

We call $z\in\hat\dC$ an \textit{exceptional point} of the meromorphic function $f$ if $\cO^-(z)$ is finite. It is not difficult to see that any meromorphic function that is not constant and not a Möbius transformation has at most two exceptional points.

\begin{lemma}   \label{lemma_blowUp}
Let $f$ be a meromorphic function that is not constant and not a Möbius transformation. If $f$ is transcendental, in addition suppose that $\cO^-(\infty)$ is finite. Let $\cK$ be a compact subset of $\dC$ that contains no exceptional point of $f$, and let $\cU\subset\dC$ be open with $\cU\cap \cJ(f)\ne\emptyset$. Then there is $n_0\in\dN$ such that 
$\cK\subset f^n(\cU)$ for all $n\ge n_0$.
\end{lemma}

This is due to Fatou for rational \cite[p.39]{fatou20} and entire functions \cite[p.356]{fatou26}. His proof for entire functions also works for transcendental meromorphic functions where $\cO^-(\infty)$ is finite. 
We also require the following result.

\begin{lemma}   \label{lemma_convToCycle}
Let $f$ be a meromorphic function that is not constant and not a Möbius transformation, and let $\cC=\{z_0,f(z_0),...,f^p(z_0)=z_0\}$ be a periodic cycle of $f$. Suppose that $z\in \cJ(f)$ is attracted by $\cC$. Then there exists $n\in\dN$ such that $f^n(z)\in \cC$.
\end{lemma}

Clearly, the hypotheses imply that $\cC\subset \cJ(f)$. It is not difficult to see that the conclusion of Lemma \ref{lemma_convToCycle} is true for repelling cycles. For rationally indifferent cycles, the result follows from the Leau flower theorem \cite[§10]{milnor06}, and for irrationally indifferent cycles, it was shown by P\'{e}rez Marco \cite{perezmarco95}.

In Lemma \ref{lemma_nodensitypoint}, we give conditions ensuring that a point $z\in\cJ(f)$ is not a point of density of $\cJ(f)$. We will then use Lemma \ref{lemma_nodensitypoint} and the Lebesgue density theorem to prove Theorem \ref{thm_zero_measure}.

\begin{lemma}   \label{lemma_nodensitypoint}
Let $f$ be a meromorphic function that is not constant and not a Möbius transformation, and let $z\in \cJ(f)\setminus\cO^{-}(\cP(f)\cup\{\infty\})$. Suppose that there exist sequences $(n_k)\in\dN^\dN$ with $\lim_{k\to\infty }n_k=\infty$ and $(r_k)\in(0,\infty)^\dN$ satisfying the following conditions:
\begin{enumerate}[(i)]
\item $\cD(f^{n_k}(z),r_k)\cap \cP(f)=\emptyset$ for all $k\in\dN$;
\item there is $\varepsilon>0$ such that $\dens(\cF(f), \cD(f^{n_k}(z),r_k))\ge\varepsilon$ for all $k\in\dN$.
\end{enumerate}
Then $z$ is not a point of density of $\cJ(f)$.
\end{lemma}

\begin{proof}
Let 
$$\omega:=\sqrt{1-\frac{\varepsilon}{2}}.$$
For $k\in\dN$, let
$$z_k:=f^{n_k}(z),\qquad \cD_k:=\cD(z_k, r_k)\qquad\text{and}\,\, \cD_k':=\cD\left(z_k,\omega r_k\right).$$
Since $\cD_k\cap \cP(f)=\emptyset$, there is a branch $\varphi_k$ of $f^{-n_k}$ defined in $\cD_k$ with $\varphi_k(z_k)=z$. By Koebe's theorems (see Lemma \ref{lemma_koebe}),
\begin{equation}  \label{eq_nodens_D}
\cD\left(z,\frac{\omega}{4}r_k|\vi_k'(z_k)|\right)\subset\vi_k(\cD_k')\subset \cD\left(z,\frac{\omega}{(1-\omega)^2}r_k|\vi_k'(z_k)|\right).
\end{equation}

We claim that 
\begin{equation}   \label{eq_lim0}
\lim_{k\to\infty}\left|\varphi_k'(z_k)\right|r_k=0.
\end{equation}

If this was not true, there would be $\delta>0$ such that $\cD(z,\delta)\subset \varphi_k(\cD_k')$ for infinitely many $k$, and hence $f^{n_k}(\cD(z,\delta))\subset \cD_k'$ for infinitely many $k$. If $f$ is transcendental and $\cO^-(\infty)$ is infinite, this is impossible because $\cO^-(\infty)$ is dense in $\cJ(f)$. Suppose that $f$ is rational or $\cO^-(\infty)$ is finite. Fix $v\in\cP(f)\cap\dC$ and let $\cK$ be of the form $\cK=\{z:\,|z-v|=\rho\}$ where $\rho$ is chosen such that $\cK$ does not contain any exceptional point of $f$. Then by Lemma \ref{lemma_blowUp}, $\cK\subset f^{n_k}(\cD(z,\delta))\subset \cD_k'\subset \cD_k$ for all large $k$. But this implies $v\in\cD_k$, contradicting (i). This proves \eqref{eq_lim0}.

We will now show that 
$$\limsup_{r\to0}\,\dens(\cF(f), \cD(z,r))>0,$$
that is, $z$ is not a point of density of $\cJ(f)$.
We have
\begin{align*}
\dens(\cF(f), \varphi_k(\cD_k'))&\ge\left(\frac{\min_{\zeta\in \cD_k'}|\varphi_k'(\zeta)|}{\max_{\zeta\in \cD_k'}|\varphi_k'(\zeta)|}\right)^2\dens(\cF(f), \cD_k')\\
&=\left(\frac{\min_{\zeta\in \cD_k'}|\varphi_k'(\zeta)|}{\max_{\zeta\in \cD_k'}|\varphi_k'(\zeta)|}\right)^2\frac{\meas(\cD_k'\cap \cF(f))}{\meas \cD_k'}\\
&\ge\left(\frac{\min_{\zeta\in \cD_k'}|\varphi_k'(\zeta)|}{\max_{\zeta\in \cD_k'}|\varphi_k'(\zeta)|}\right)^2\cdot\frac{\meas(\cD_k\cap \cF(f))-\meas (\cD_k\setminus \cD_k')}{\meas \cD_k}.
\end{align*}
Hence, by the Koebe distortion theorem and (ii),
\begin{equation}  \label{eq_nodens_dens}
\dens(\cF(f),\vi_k(\cD_k'))\ge \left(\frac{1-\omega}{1+\omega}\right)^8\cdot\left(\varepsilon-\frac{\pi r_k^2-\pi r_k^2\omega^2}{\pi r_k^2}\right)
=\left(\frac{1-\omega}{1+\omega}\right)^8\cdot\frac{\varepsilon}{2}.
\end{equation}
By \eqref{eq_nodens_dens} and \eqref{eq_nodens_D},
\begin{equation}
\begin{split}
&\dens\left(\cF(f), \cD\left(z,\frac{\omega}{(1-\omega)^2}|\varphi_k'(z_k)|r_k\right)\right)\\
&\ge\dens(\cF(f), \varphi_k(\cD_k'))\cdot\dens\left(\vi_k(\cD_k'), \cD\left(z,\frac{\omega}{(1-\omega)^2}|\varphi_k'(z_k)|r_k\right)\right)\\
&\ge\left(\frac{1-\omega}{1+\omega}\right)^8\frac{\epsi}{2}\cdot\frac{1}{16}\left(1-\omega\right)^4.\qedhere
\end{split}
\end{equation}
\end{proof}

\begin{proof}[Proof of Theorem \ref{thm_zero_measure}]
We first show that $\meas(\cP(f)\cap\cJ(f))$ is zero. In order to do so, write $\cP(f)\cap\cJ(f)\cap\dC=\cP_1\cup\cP_2$ with $\cP_1:=\cP(f)\cap\cJ(f)\cap\overline{\cD(0,R_1)}$ and $\cP_2:=(\cP(f)\cap\cJ(f)\cap\dC)\setminus\overline{\cD(0,R_1)}$. Since $\cP_1$ is a finite set, we only have to show that $\meas(\cP_2)=0$.

Since $\cJ(f)$ is uniformly thin at $(\cP(f)\cap\dC)\setminus\overline{\cD(0,R_1)}$, there are $\delta_1,\epsi_1>0$ such that for all $v\in\cP(f)\cap\dC$ with $|v|> R_1$ and all $\zeta\in\overline{\cD(v,\delta_1)}$, we have 
\begin{equation}  \label{eq_thinatv}
\dens(\cF(f),\cD(\zeta,|\zeta-v|))>\epsi_1.
\end{equation}
Let $z\in\cP_2$ and $r\in(0,2\delta_1)$. Then $\cD(z+r/2,r/2)\subset\cD(z,r)$ and $\dens(\cF(f),\cD(z+r/2,r/2))>\epsi_1$. Thus, 
$$\dens(\cF(f),\cD(z,r))\ge\dens\left(\cF(f),\cD\left(z+\frac{r}{2},\frac{r}{2}\right)\right)\cdot\dens\left(\cD\left(z+\frac{r}{2},\frac{r}{2}\right),\cD(z,r)\right)>\frac{\epsi_1}{4}.$$
Hence, $z$ is not a point of density of $\cJ(f)$. By the Lebesgue density theorem (see, e.g., \cite[Corollary 2.14]{mattila95}), the Lebesgue measure of $\cP_2$ is zero. So $\meas(\cP(f)\cap\cJ(f)\cap\dC)=0$ and hence also $\meas(\cO^-(\cP(f)\cap\cJ(f)\cap\dC))=0$.
Since $\cO^-(\infty)$ is countable, we obtain that $\meas(\cO^-((\cP(f)\cap\cJ(f))\cup\{\infty\}))=0$.

Next, we show that each $z\in\cJ(f)\setminus\cO^-(\cP(f)\cup\{\infty\})$ satisfies 
\begin{equation}  \label{eq_disttoP1}
\limsup_{n\to\infty}\dist(f^n(z),\cP_1)>0.
\end{equation} 
In order to do so, suppose that $\lim_{n\to\infty}\dist(f^n(z),\cP_1)=0$. We show that then $z\in\cO^-(\cP_1)$, contradicting our assumption.

Because $\cP_1$ is finite, there is a subsequence, $(f^{n_k}(z))$, that converges to some $w\in \cP_1$. For all $j\in\dN$, we have $f^j(w)=\lim_{k\to\infty}f^{n_k+j}(z)\in\cP_1$. Thus, $w$ is preperiodic, that is, $f^l(w)$ is periodic for some $l\in\dN$. Assume without loss of generality that $l=0$, that is, there is $p\in\dN$ with $f^p(w)=w.$ 

Let $\alpha>0$ such that the disks $\overline{\cD(\zeta,\alpha)}$ with $\zeta\in \cP_1$ are pairwise disjoint, and let $\beta\in(0,\alpha)$ such that $f(\cD(f^j(w),\beta))\subset \cD(f^{j+1}(w),\alpha)$ for all $j\in \{0,...,p-1\}.$ Then by periodicity, this is true for all $j\ge0$. For large $k$, we have $\dist(f^{n_k+j}(z),\cP_1)<\beta$ for all $j\ge0$ and $f^{n_k}(z)\in \cD(w,\beta)$. Then $f^{n_k+1}(z)\in \cD(f(w),\alpha)$. Since the disks $\overline{\cD(\zeta, \alpha)}$ with $\zeta\in \cP_1$ are disjoint, we have $|f^{n_k+1}(z)-\zeta|>\alpha>\beta$ for all $\zeta\in \cP_1\setminus\{f(w)\}$. Thus, $|f^{n_k+1}(z)-f(w)|<\beta$. Inductively, we obtain that $f^{n_k+j}(z)\in \cD(f^j(w), \beta)$ for all $j\in\dN$. Thus, $f^n(z)$ is attracted by the cycle $\{w,f(w),...,f^{p-1}(w),f^p(w)=w\}$. By Lemma \ref{lemma_convToCycle}, $z$ is eventually mapped to this cycle, so $z\in\cO^-(\cP_1)$. 

Now let $z\in \cJ(f)\setminus\cO^-(\cP(f)\cup\{\infty\})$. By \eqref{eq_disttoP1}, there exist a subsequence $(f^{n_k}(z))$ and $\eta>0$ such that 
\begin{equation}   \label{eq_distP1}
\dist(f^{n_k}(z), \cP_1)>\eta
\end{equation}
for all $k\in\dN$. We will show that $z$ satisfies the assumptions of Lemma \ref{lemma_nodensitypoint} and hence is not a point of density of $\cJ(f)$. Let 
$$d_k:=\dist(f^{n_k}(z), \cP(f)),$$
and let $z_k\in \cP(f)$ with
$$|f^{n_k}(z)-z_k|=d_k.$$
First suppose that the sequence $(d_k)$ is bounded, say $d_k\le\gamma$ for all $k$. 
We distinguish three cases.

\textit{1st case:} $|z_k|\le R_1$ for infinitely many $k$. 
By passing to a subsequence if necessary, we can assume that $|z_k|\le R_1$ for all $k$. Then $(f^{n_k}(z))$ is bounded, and by again passing to a subsequence, we can assume that $f^{n_k}(z)$ converges to some $w\in \cJ(f)$. By \eqref{eq_distP1}, we have $w\notin \cP_1$. If $|w|>R_1$, then for large $k$, we have
$$d_k=|f^{n_k}(z)-z_k|\ge|f^{n_k}(z)|-|z_k|\ge|f^{n_k}(z)|-R_1>|f^{n_k}(z)-w|.$$
 Thus, $w\notin\cP(f)$, so $\nu:=\dist(w, \cP(f))>0$. For large $k$,
$$\cD\left(w,\frac{\nu}{4}\right)\subset \cD\left(f^{n_k}(z),\frac{\nu}{2}\right)\subset \cD(w,\nu).$$
Thus, $\cD(f^{n_k}(z),\nu/2)\cap \cP(f)=\emptyset$, and
\begin{equation}
\begin{split}
&\dens\left(\cF(f), \cD\left(f^{n_k}(z),\frac{\nu}{2}\right)\right)\\
&\ge\dens\left(\cD\left(w,\frac{\nu}{4}\right),\cD\left(f^{n_k}(z),\frac{\nu}{2}\right)\right)\cdot\dens\left(\cF(f), \cD\left(w,\frac{\nu}{4}\right)\right)\\
&=\frac{1}{4}\dens\left(\cF(f), \cD\left(w, \frac{\nu}{4}\right)\right)>0.
\end{split}
\end{equation} 
By Lemma \ref{lemma_nodensitypoint}, $z$ is not a point of density of $\cJ(f)$.

\textit{2nd case:} $|z_k|> R_1$ and $d_k\le\delta_1$ for infinitely many $k$, without loss of generality for all $k$. Then by \eqref{eq_thinatv},
$$\dens(\cF(f), \cD(f^{n_k}(z), d_k))>\varepsilon_1.$$
By Lemma \ref{lemma_nodensitypoint}, $z$ is not a point of density of $\cJ(f)$.

\textit{3rd case:} $|z_k|> R_1$ and $d_k>\delta_1$ for infinitely many $k$, without loss of generality for all $k$. Let
$$w_k:=z_k+\frac{\delta_1}{d_k}(f^{n_k}(z)-z_k).$$
Then
$$|f^{n_k}(z)-w_k|=\left(1-\frac{\delta_1}{d_k}\right)|f^{n_k}(z)-z_k|=d_k-\delta_1$$
and hence
$$\cD(w_k,\delta_1)\subset \cD(f^{n_k}(z), d_k).$$
Also, $|w_k-z_k|=\delta_1$. By the hypotheses,
\begin{align*}
\dens(\cF(f), \cD(f^{n_k}(z),d_k))&\ge\dens(\cD(w_k,\delta_1),\cD(f^{n_k}(z),d_k))\cdot\dens(\cF(f), \cD(w_k,\delta_1))\\
&\ge\frac{\delta_1^2}{d_k^2}\epsi_1\ge\frac{\delta_1^2}{\gamma^2}\varepsilon_1.
\end{align*}
By Lemma \ref{lemma_nodensitypoint}, $z$ is not a point of density of $\cJ(f)$. 

Now suppose that the sequence $(d_k)$ is unbounded. Since $\cJ(f)$ is thin at $\infty$, there are $R_0,\epsi_0>0$ such that for all $v\in\dC$, we have
$$\dens(\cF(f),\cD(v,R_0))>\epsi_0.$$
 By passing to a subsequence if necessary, we can assume that
$d_k\ge R_0$ for all $k$. Then $\cD(f^{n_k}(z),R_0)\cap \cP(f)=\emptyset$ for all $k$. Also, 
$$\dens(\cF(f), \cD(f^{n_k}(z),R_0))>\varepsilon_0.$$ 
By Lemma \ref{lemma_nodensitypoint}, $z$ is not a point of density of $\cJ(f)$. 

Altogether, it follows that the set of density points of $\cJ(f)$ has Lebesgue measure zero. By the Lebesgue density theorem, the Lebesgue measure of $\cJ(f)$ is zero.
\end{proof}


\section{A change of variables}    \label{sec_changeOfVar}

Throughout the remaining part of the paper, let $g$ be defined by \eqref{eq_g}, that is, 
$$g(z)=\int_0^zp(t)e^{q(t)}\,dt+c.$$

\begin{remark}
Suppose that $q(t)=at^d+O(t^{d-1})$ as $t\to\infty$, where $a\in\dC\setminus\{0\}$ and $d\ge1$. Let $\alpha\in\dC$ with $\alpha^d=a$. Then $q(t/\alpha)=t^d+O(t^{d-1})$ as $t\to\infty$, 
$$g(z/\alpha)=\int_0^{z/\alpha}p(t)e^{q(t)}\,dt+c=\int_0^z\frac{1}{\alpha}p\left(\frac{t}{\alpha}\right)e^{q\left(t/\alpha\right)}\,dt+c,$$
and Newton's method for $g(z/\alpha)$ is conjugate to Newton's method for $g$ via $z\mapsto\alpha z$. Thus, we can and will assume without loss of generality that $a=1$, that is,
$$q(t)=t^d+O(t^{d-1})$$
as $t\to\infty.$

 Also, since the functions $g$ and $b\cdot g$ for $b\in\dC\setminus\{0\}$ have the same zeros and Newton's methods for $g$ and $b\cdot g$ coincide, we can and will assume without loss of generality that $p$ has the form 
 $$p(t)=dt^m+O(t^{m-1})$$
  as $t\to\infty$, where $d=\deg(q)$.
\end{remark}

Let $f$ be defined by \eqref{eq_f}, that is, $f$ is the Newton map corresponding to $g$. In order to prove Theorem \ref{thm_main}, it will be useful to consider the change of variables $w=q(z)$. Let $R>0$ such that all critical values of $q$ are contained in $\cD(0,R)$ and such that for $|z|\ge(1/2)R^{1/d}$, we have
\begin{equation}   \label{eq_qEstimate}
\frac{1}{2^d}|z|^d\le|q(z)|\le2^d|z|^d.
\end{equation}
Define
$$\cG:=\dC\setminus(\overline{\cD(0,R)}\cup[0,\infty)).$$

\begin{lemma}   \label{lemma_change_of_var}
There exists $c>0$ such that the set $q^{-1}(\cG)$ consists of $d$ components, $\cS_1,..., \cS_d$, satisfying
$$\cS_j\subset\left\{z:\,|z|>\frac{1}{2}R^{1/d},\,\frac{2(j-1)\pi}{d}-\frac{c}{|z|}<\arg(z)<\frac{2j\pi}{d}+\frac{c}{|z|}\right\}$$
and
$$\cS_j\supset\left\{z:\,|z|>2R^{1/d},\,\frac{2(j-1)\pi}{d}+\frac{c}{|z|}<\arg(z)<\frac{2j\pi}{d}-\frac{c}{|z|}\right\}$$
for all $j\in\{1,...,d\}$. Moreover, $q$ maps each $\cS_j$ conformally onto $\cG$.
\end{lemma}

\begin{proof}
Since $\cG$ is simply connected and contains no critical values of $q$, its preimage $q^{-1}(\cG)$ consists of $d$ components, and $q$ maps each of them conformally onto $\cG$. By \eqref{eq_qEstimate}, 
$$q\left(\cD\left(0,\frac{1}{2}R^{1/d}\right)\right)\subset \cD(0,R)$$
and
$$q\left(\dC\setminus \cD\left(0,2R^{1/d}\right)\right)\subset \dC\setminus \cD(0,R).$$
Also, for $z\in\dC$, we have
$$\arg(q(z))=\arg\left[z^d\left(1+\cO\left(\frac{1}{z}\right)\right)\right] 
\equiv d\arg(z)+\arg\left(1+\cO\left(\frac{1}{z}\right)\right)
\equiv d\arg(z)+\cO\left(\frac{1}{z}\right)$$
mod $2\pi$ as $z\to\infty$.
Thus,
$$\arg(z)\equiv\frac{\arg(q(z))}{d}+\cO\left(\frac{1}{z}\right)\mod\frac{2\pi}{d}$$
as $z\to\infty$. Using that $q$ is surjective, we obtain the desired conclusion.
\end{proof}

For $j\in\{1,...,d\}$, let $\vi_j$ be the branch of $q^{-1}$ defined in $\cG$ with $\vi_j(\cG)=\cS_j$.


\section{The asymptotics of \texorpdfstring{$g$ and $f$}{g and f}}  \label{sec_asymptotics}

In this section, we give asymptotic representations for $g(\vi_j(w)),g(z),f(\vi_j(w)),f(z)$.
Let
\begin{equation}   \label{eq_lambda}
\lambda:=\frac{d-1-m}{d}.
\end{equation}
Then 
\begin{equation}  \label{eq_pq'}
\frac{p(z)}{q'(z)}=z^{-\lambda d}\left(1+O\left(\frac{1}{z}\right)\right)
\end{equation}
as $z\to\infty$ and, for $j\in\{1,...,d\}$,
\begin{equation}  \label{eq_pq'vi}
\left|\frac{p(\vi_j(w))}{q'(\vi_j(w))}\right|=|w|^{-\lambda}\left(1+O\left(\frac{1}{|w|^{1/d}}\right)\right)
\end{equation}
as $w\to\infty$ in $\cG$.

\begin{lemma}   \label{lemma_asymptotics_g}
Let $j\in\{1,...,d\}$. Then there exists $c_j\in\dC$ such that 
$$g(\varphi_j(w))=c_j+\frac{p(\vi_j(w))}{q'(\vi_j(w))}\left(1+\frac{\lambda}{w}+O\left(\frac{1}{|w|^{1+1/d}}\right)\right)e^w$$
as $w\to\infty$ in $\cG$.
\end{lemma}

In terms of $z=\vi_j(w)$, Lemma \ref{lemma_asymptotics_g} says the following.

\begin{cor}   \label{cor_asymptotics_g}
For $j\in\{1,...,d\}$, we have
$$g(z)=c_j+\frac{p(z)}{q'(z)}\left(1+\frac{\lambda}{z^d}+O\left(\frac{1}{z^{d+1}}\right)\right)e^{q(z)}$$
as $z\to\infty$ in $\cS_j$.
\end{cor}

\begin{proof}[Proof of Lemma \ref{lemma_asymptotics_g}]
Let $x_0\in (-\infty,-R)=\cG\cap(-\infty,0]$ and $w\in\cG$. Then
\begin{equation}
\begin{split}
g(\vi_j(w))&=\int_0^{\vi_j(w)}p(t)e^{q(t)}\,dt+c\\
&=\int_{\vi_j(x_0)}^{\vi_j(w)}p(t)e^{q(t)}\,dt+\int_0^{\vi_j(x_0)}p(t)e^{q(t)}\,dt+c\\
&=\int_{\vi_j(x_0)}^{\vi_j(w)}p(t)e^{q(t)}\,dt+g(\vi_j(x_0))\\
&=\int_{x_0}^w\vi_j'(s)p(\vi_j(s))e^s\,ds+g(\vi_j(x_0)).
\end{split}
\end{equation}

Let 
$$r(s):=\vi_j'(s)p(\vi_j(s))=\frac{p(\vi_j(s))}{q'(\vi_j(s))}.$$
Repeated integration by parts yields
\begin{equation}
\int_{x_0}^wr(s)e^s\,ds=\left(r(s)-r'(s)+r''(s)\right)e^s\Big|_{x_0}^w-\int_{x_0}^wr'''(s)e^s\,ds.
\end{equation}
We have
\begin{equation}
\begin{split}
r'(s)&=\vi_j'(s)\frac{q'(\vi_j(s))p'(\vi_j(s))-q''(\vi_j(s))p(\vi_j(s))}{q'(\vi_j(s))^2}\\
&=\left(\frac{1}{q'(\vi_j(s))}\cdot\frac{p(\vi_j(s))}{q(\vi_j(s))}\right)\cdot\left(\frac{q(\vi_j(s))}{p(\vi_j(s))}\cdot\frac{q'(\vi_j(s))p'(\vi_j(s))-q''(\vi_j(s))p(\vi_j(s))}{q'(\vi_j(s))^2}\right)\\
&=\frac{p(\vi_j(s))}{q'(\vi_j(s))s}\cdot\frac{q(\vi_j(s))q'(\vi_j(s))p'(\vi_j(s))/p(\vi_j(s))-q(\vi_j(s))q''(\vi_j(s))}{q'(\vi_j(s))^2}\\
&=\frac{r(s)}{s}\cdot\frac{m-(d-1)}{d}\left(1+O\left(\frac{1}{|s|^{1/d}}\right)\right)\\
&=-\frac{\lambda}{s}r(s)\left(1+O\left(\frac{1}{|s|^{1/d}}\right)\right).
\end{split}
\end{equation}
Also, a computation shows that
$$r''(s)=r(s)O\left(\frac{1}{s^2}\right)\quad\text{and}\,\,r'''(s)=r(s)O\left(\frac{1}{s^3}\right)$$
as $s\to\infty$. With $h(x_0):=(r(x_0)-r'(x_0)+r''(x_0))e^{x_0}$, we obtain
$$\int_{x_0}^wr(s)e^s\,ds=r(w)e^w\left(1+\frac{\lambda}{w}+O\left(\frac{1}{|w|^{1+1/d}}\right)\right)-h(x_0)-\int_{x_0}^wr'''(s)e^s\,ds.$$
We have
$$\int_{x_0}^wr'''(s)e^s\,ds=\int_{-|w|}^wr'''(s)e^s\,ds+\int_{-\infty}^{-|w|}r'''(s)e^s\,ds-\int_{-\infty}^{x_0}r'''(s)e^s\,ds.$$
To estimate $\int_{-|w|}^wr'''(s)e^s\,ds$, let $\gamma$ be the part of the circle with centre $0$ and radius $|w|$ that connects $-|w|$ and $w$ in $\cG$. Then $\re s\le\re w$ for $s\in\gamma$. We obtain
\begin{equation}
\begin{split}
\left|\int_{-|w|}^wr'''(s)e^s\,ds\right|&\le\length(\gamma)\cdot\max_{s\in\gamma}\left|r'''(s)e^s\right|\le O(|w|)|r(w)|O\left(\frac{1}{|w|^3}\right)e^{\re w}\\
&=|r(w)|O\left(\frac{1}{|w|^2}\right)|e^w|.
\end{split}
\end{equation}
Let us now estimate $\int_{-\infty}^{-|w|}r'''(s)e^s\,ds$. By \eqref{eq_pq'vi}, we have $|r(s)|\sim |s|^{-\lambda}$ as $|s|\to\infty$. First suppose that $\lambda\ge0$.  Using that $r'''(s)=r(s)O(1/s^3)$, we obtain
\begin{equation}
\begin{split}
\left|\int_{-\infty}^{-|w|}r'''(s)e^s\,ds\right|&\le |r(w)|e^{-|w|}O\left(\frac{1}{|w|^3}\right)\int_{-\infty}^{-|w|}e^{s+|w|}\,ds\\
&\le|r(w)e^w|O\left(\frac{1}{|w|^3}\right)\int_{-\infty}^0e^s\,ds=|r(w)e^w|O\left(\frac{1}{|w|^3}\right).
\end{split}
\end{equation}
Now suppose that $\lambda<0$. Then
$$\left|\int_{-\infty}^{-|w|}r'''(s)e^s\,ds\right|\le O\left(\frac{1}{|w|^3}\right)\int_{-\infty}^{-|w|}|s|^{-\lambda}e^s\,ds.$$
Integration by parts yields
$$\int_{-\infty}^{-|w|}|s|^{-\lambda}e^s\,ds=O(|w|^{-\lambda}e^{-|w|})\le O(|r(w)e^w|)$$
and hence
$$\int_{-\infty}^{-|w|}r'''(s)e^s\,ds=r(w)e^wO\left(\frac{1}{|w|^3}\right).$$
Altogether, we obtain the desired conclusion with
$$c_j=g(\vi_j(x_0))-h(x_0)+\int_{-\infty}^{x_0}r'''(s)e^s\,ds.\qedhere$$
\end{proof}

For the function $f$ from Newton's method for $g$, Lemma \ref{lemma_asymptotics_g} yields the following.

\begin{cor} \label{cor_asymptotics_f}
For $j\in\{1,...,d\}$, we have
$$f(\vi_j(w))=\vi_j(w)-\frac{1}{q'(\vi_j(w))}\left(1+\frac{\lambda}{w}+O\left(\frac{1}{|w|^{1+1/d}}\right)\right)-\frac{c_je^{-w}}{p(\vi_j(w))}$$
as $w\to\infty$ in $\cG$.
\end{cor}

In terms of $z$, Corollary \ref{cor_asymptotics_f} says the following.

\begin{cor}   \label{cor_asymptotics_fz}
For $j\in\{1,...,d\}$, we have
$$f(z)=z-\frac{1}{q'(z)}\left(1+\frac{\lambda}{z^d}+O\left(\frac{1}{|z|^{d+1}}\right)\right)-\frac{c_je^{-q(z)}}{p(z)}$$
as $z\to\infty$ in $\cS_j$.
\end{cor}


\section{Partitioning the plane}   \label{sec_partition}

For a more detailed study of the behaviour of $f\circ\vi_j$, we will divide the complex plane into several sets depending on how large $|e^{-w}|$ is compared to some power of $|w|$.
More precisely, we consider sets whose boundary points satisfy
\begin{equation}  \label{eq_boundarypoints}
\re w=\mu\log|w|-\log\alpha
\end{equation}
for certain $\mu\in\dR$ and $\alpha>0$. Such sets were also considered by Jankowski \cite{jankowski96}. In this section, we will show that given $\mu\in\dR$, $\alpha>0$ and $y\in\dR$ of sufficiently large modulus, there is a unique $x_y\in\dR$ such that $w=x_y+iy$ satisfies \eqref{eq_boundarypoints}. We also give a proof of several properties of the mapping $y\mapsto x_y$ which in part are also shown in \cite[§3.3.4]{jankowski96}.

\begin{lemma}  \label{lemma_part_exist}
Let $\mu\in\dR$, $\alpha>0$ and $y\in\dR$ with $|y|\ge2|\mu|$. Then there exists a unique $x_y\in\dR$ with
\begin{equation}   \label{eq_part_xy}
x_y =\mu\log|x_y+iy|-\log\alpha.
\end{equation}
If $x>x_y$, then 
\begin{equation}  \label{eq_incr1}
x>\mu\log|x+iy|-\log\alpha.
\end{equation}
If $x<x_y$, then 
\begin{equation}  \label{eq_incr2}
x<\mu\log|x+iy|-\log\alpha.
\end{equation}
\end{lemma}

\begin{proof}
Let $\vi:\dR\to\dR$,
$$\vi(x)=x-\mu\log|x+iy|=x-\frac{\mu}{2}\log(x^2+y^2).$$
Then $\vi(x)\to\infty$ as $x\to\infty$, and $\vi(x)\to-\infty$ as $x\to-\infty$. Thus, $\vi$ is surjective, so there exists $x_y$ satisfying \eqref{eq_part_xy}.

Also,
$$\vi'(x)=1-\frac{\mu x}{x^2+y^2}.$$
Since
$$\frac{|\mu x|}{x^2+y^2}\le|\mu|\frac{\max\{|x|,|y|\}}{\max\{x^2,y^2\}}=\frac{|\mu|}{\max\{|x|,|y|\}}\le\frac{1}{2},$$
we have
$$\vi'(x)\ge\frac{1}{2}.$$
Thus, $\vi$ is strictly increasing, which implies \eqref{eq_incr1} and \eqref{eq_incr2}. In particular, $\vi$ is injective, so $x_y$ is unique.
\end{proof}

For $\mu\in\dR$ and $\alpha>0$, let
$$\gamma_{\mu,\alpha}:(-\infty,-2|\mu|]\cup[2|\mu|,\infty)\to\dR,\,\gamma_{\mu,\alpha}(y)=x_y.$$

\begin{lemma}   \label{lemma_gamma}
Let $\mu\in\dR$ and $\alpha>0$. 
\begin{enumerate}[(i)]
\item The function $\gamma_{\mu,\alpha}$ is continuously differentiable.
\item If $\mu>0$, then  $\lim_{|y|\to\infty}\gamma_{\mu,\alpha}(y)=\infty.$
If $\mu<0$, then
$\lim_{|y|\to\infty}\gamma_{\mu,\alpha}(y)=-\infty.$ For $\mu=0$, $\gamma_{\mu,\alpha}\equiv-\log\alpha$.
\item $|\gamma_{\mu,\alpha}'(y)|\le2|\mu|/|y|$. In particular, $\lim_{|y|\to\infty}\gamma_{\mu,\alpha}'(y)=0$ uniformly in $\alpha$.
\item For $\alpha>\beta>0$, we have
$$\frac{2}{3}\log\frac{\alpha}{\beta}\le\gamma_{\mu,\beta}(y)-\gamma_{\mu,\alpha}(y)\le2\log\frac{\alpha}{\beta}$$
and
$$\lim_{|y|\to\infty}(\gamma_{\mu,\beta}(y)-\gamma_{\mu,\alpha}(y))=\log\frac{\alpha}{\beta}.$$
\end{enumerate}
\end{lemma}

\begin{proof}
For $\mu=0$, the results are obvious. We will prove the Lemma for $\mu>0$, the proof for $\mu<0$ is analogous. To prove (i)-(iii), note that the condition
$$x=\mu\log|x+iy|-\log\alpha$$
is equivalent to
$$y^2=\alpha^{2/\mu}e^{(2/\mu) x}-x^2.$$
The function 
$$\psi(x)=\alpha^{2/\mu}e^{(2/\mu) x}-x^2$$
 satisfies 
\begin{equation}   \label{eq_part_psi}
\lim_{x\to-\infty}\psi(x)=-\infty\quad\text{and}\,\,\lim_{x\to\infty}\psi(x)=\infty.
\end{equation}
Let $x_0:=\max\{x:\,\psi(x)=4\mu^2\}$. Then $\psi(x)>4\mu^2$ for $x>x_0$. Also,
\begin{equation}    \label{eq_part_psiprime}
\psi'(x)=\frac{2}{\mu}\alpha^{2/\mu}e^{(2/\mu) x}-2x=\frac{2}{\mu}(\psi(x)+x^2-\mu x).
\end{equation}
It is not difficult to see that
$x^2-\mu x\ge-\mu^2/4$ for all $x\in\dR$. Thus,
\begin{equation}  \label{eq_part_deriv}
\psi'(x)\ge\frac{2}{\mu}\left(\psi(x)-\frac{\mu^2}{4}\right)>0
\end{equation}
 for $x>x_0$. In particular, $\psi:[x_0,\infty)\to[4\mu^2,\infty)$ is bijective. This implies that 
$$\gamma_{\mu,\alpha}(y)=\psi^{-1}(y^2)$$
is a continuously differentiable function. By \eqref{eq_part_psi}, (ii) is satisfied. Also, by \eqref{eq_part_deriv} and since $y^2\ge4\mu^2$, we have
\begin{equation}
|\gamma_{\mu,\alpha}'(y)|=\left|\frac{2y}{\psi'(\psi^{-1}(y^2))}\right|\le\frac{2|y|}{(2/\mu)(\psi(\psi^{-1}(y^2))-\mu^2/4)}=\frac{\mu |y|}{y^2-\mu^2/4}\le\frac{2\mu}{|y|},
\end{equation}
that is, (iii) is satisfied. To prove (iv), let $y\in\dR$ with $|y|\ge2\mu$ be fixed, and let $\vi$ be as in the proof of Lemma \ref{lemma_part_exist}. Let $x_1:=\gamma_{\mu,\alpha}(y)$ and $x_2:=\gamma_{\mu,\beta}(y).$ Then by the mean value theorem,
$$\log\frac{\alpha}{\beta}=\vi(x_2)-\vi(x_1)=\vi'(\xi)(x_2-x_1)$$
for some $\xi\in[x_2,x_1]$. In the proof of Lemma \ref{lemma_part_exist}, we have seen that $\vi'(\xi)\ge1/2$, and the same arguments show that $\vi'(\xi)\le3/2$. Also, $\vi'(\xi)\to1$ as $|y|\to\infty$. 
\end{proof}

For $\mu\in\dR,\,\alpha>0$ and $\nu\ge2|\mu|$, define
\begin{equation}
\begin{split}
\cH(\mu,\alpha,\nu)&:=\{w:\,\re w\ge\mu\log|w|-\log(\alpha),\,|\im w|\ge\nu\}\\
&=\{x+iy:\,|y|\ge\nu,\,x\ge\gamma_{\mu,\alpha}(y)\}.
\end{split}
\end{equation}

Also, let 
\begin{equation}
\begin{split}
\Gamma(\mu,\alpha)&:=\{w:\,|\im w|\ge 2|\mu|,\,\re w=\mu\log|w|-\log\alpha\}\\
&=\{\gamma_{\mu,\alpha}(y)+iy:\,|y|\ge2|\mu|\}.
\end{split}
\end{equation}

\begin{remark}   \label{remark_gamma}
Note that if $w\in\Gamma(\mu,\alpha)$, then
$$|e^{-w}|=e^{-\re w}=\alpha|w|^{-\mu};$$
if $w\in\cH(\mu,\alpha,\nu)$, then
$$|e^{-w}|\le\alpha|w|^{-\mu};$$
and if $w\in\dC\setminus\cH(\mu,\alpha,\nu)$ with $|\im w|\ge\nu$, then
$$|e^{-w}|>\alpha|w|^{-\mu}.$$
\end{remark}


\section{The singular values of \texorpdfstring{$f$}{f}}   \label{sec_singularities}

Recall that 
$$g(z)=\int_0^zp(t)e^{q(t)}\,dt+c$$
where $p(t)=t^d+O(t^{d-1})$ and $q(t)=dt^m+O(t^{m-1})$ as $t\to\infty$, and $f$ is the Newton map corresponding to $g$. Let us assume throughout the rest of the paper that $g$ and $f$ satisfy the assumptions of Theorem \ref{thm_main}.

In this section, we determine the location of the singular values of $f$.

\begin{lemma}
\textnormal{\cite[§7, p.238]{bergweiler93}} The function $f$ does not have finite asymptotic values.
\end{lemma}

So each singular value of $f$ in $\dC$ must be a critical value or a limit point of critical values. We have
$$f'(z)=\frac{g(z)g''(z)}{g'(z)^2}.$$
Thus, the critical points of $f$ are:
\begin{enumerate}
\item the zeros of $g$ that are not zeros of $g'$. These are superattracting fixed points of $f$ and form a discrete subset of $\dC$.
\item the zeros of $g''$ that are not zeros of $g$ or $g'$. There are only finitely many of these, $z_1,...,z_N$, and by assumption, each $z_j$ is attracted by a periodic cycle.
\end{enumerate}
In particular, the set of critical values of $f$ does not have limit points in $\dC$. So every singular value of $f$ in $\dC$ is a critical value, and all but finitely many of them are superattracting fixed points.

\begin{lemma}   \label{lemma_sing_assumptions}
The set $\cP(f)\cap\cJ(f)$ is finite.
\end{lemma}

\begin{proof}
Since the superattracting fixed points of $f$ form a discrete subset of $\dC$, the set $\cP(f)\cap\cJ(f)$ is contained in the closure of $\cO^+(\{z_1,...,z_N\})$. Each $z_j$ is attracted by a periodic cycle $\cC$. In particular, $\cO^+(z_j)$ is bounded and has only finitely many limit points. If $z_j\in\cJ(f)$, then Lemma \ref{lemma_convToCycle} yields that $z_j$ is eventually mapped to $\cC$, so the forward orbit of $z_j$ is finite.
\end{proof}

By Corollary \ref{cor_asymptotics_g},
$$g(z)=c_j+\frac{p(z)}{q'(z)}\left(1+\frac{\lambda}{z^d}+O\left(\frac{1}{z^{d+1}}\right)\right)e^{q(z)}$$
as $z\to\infty$ in $\cS_j$, for $j\in\{1,...,d\}$. We will see later that if $c_j\ne0$, then $g$ has infinitely many zeros in $\cS_j$. It is easy to see that this cannot be the case for $c_j=0$. However, we will show now that under the assumptions of Theorem \ref{thm_main}, the case $c_j=0$ does not occur.

\begin{lemma}
If $c_j=0$ for some $j\in\{1,...,d\}$, then $f$ has a Baker domain.
\end{lemma}

\begin{proof}
If $c_j=0$, then Corollary \ref{cor_asymptotics_fz} yields that 
$$f(z)=z-\frac{1}{dz^{d-1}}+O\left(\frac{1}{z^d}\right)$$
as $z\to\infty$ in $\cS_j$. The claim now follows from \cite[§8, §11]{fatou19} (see also \cite[Theorem 2]{hinkkanen92}).
\end{proof}

\begin{cor}
If the assumptions of Theorem \ref{thm_main} are satisfied, then $c_j\ne0$ for all $j\in\{1,...,d\}$.
\end{cor}

\begin{proof}
A theorem by Bergweiler \cite[Theorem 2]{bergweiler93} says that if $g$ and $f$ are defined by \eqref{eq_g} and \eqref{eq_f}, then every cycle of Baker domains of $f$ contains a singularity of $f^{-1}$. This cannot be true under the assumptions of Theorem \ref{thm_main}.
\end{proof}

We now investigate the location of the zeros of $g$. It turns out that the images under $q$ of all but finitely many of them are close to the curves $\Gamma(\lambda,1/|c_j|)$ defined in Section \ref{sec_partition}. More precisely, we have the following.

\begin{lemma}   \label{lemma_zeros}
For $j\in\{1,...,d\}$ and $k\in\dZ$, let $v_{j,k}\in\Gamma(\lambda,1/|c_j|)$ such that 
$$\im v_{j,k}=\begin{cases}\arg(-c_j)+\lambda(\pi/2+2\pi(j-1))+2k\pi&\text{if }k\ge0\\
\arg(-c_j)+\lambda(-\pi/2+2\pi j)+2k\pi&\text{if }k<0.\end{cases}$$
 If $z\in \cS_j$ is a zero of $g$ and $|z|$ is large, then there exists $k\in\dZ$ such that
\begin{equation}   \label{eq_zero}
q(z)=v_{j,k}+o(1)
\end{equation}
as $|z|\to\infty$. Vice versa, if $j\in\{1,...,d\}$ and $|k|$ is large, then $g$ has a zero $z\in \cS_j$ satisfying \eqref{eq_zero}.
\end{lemma}

\begin{proof}
First suppose that $z\in \cS_j$ is a zero of $g$. By Corollary \ref{cor_asymptotics_g} and \eqref{eq_pq'},
$$g(z)=c_j+z^{-d\lambda}(1+o(1))e^{q(z)}$$
as $z\to\infty$, and hence
\begin{equation}  \label{eq_zeros}
e^{q(z)}=-c_jz^{d\lambda}(1+o(1)).
\end{equation}
Thus,
\begin{equation}
\begin{split}
\re q(z)&=\log\left|e^{q(z)}\right|=\log|c_j|+d\lambda\log|z|+o(1)\\
&=\log|c_j|+\lambda\log|q(z)|+o(1)\\
&=\lambda\log|q(z)|-\log\frac{1}{|c_j|}+o(1).
\end{split}
\end{equation}
In particular, $\re q(z)=o(|q(z)|)$ as $z\to\infty$ and hence
\begin{equation}  \label{eq_argq}
\arg q(z)=\pm\frac{\pi}{2}+o(1).
\end{equation}
Let us now assume that $\im q(z)>0$ and hence $\arg q(z)=\pi/2+o(1).$ The proof in the case where $\im q(z)<0$ is analogous.
By \eqref{eq_zeros},
\begin{equation}   \label{eq_zeros_im}
\im q(z)\equiv\arg(-c_j)+d\lambda\arg(z)+o(1) \mod 2\pi.
\end{equation}
We have 
 $$\arg q(z) =\arg(z^d(1+o(1)))\equiv d\arg z+o(1)\mod 2\pi$$
and hence
\begin{equation}
\arg z \equiv\frac{1}{d}\arg q(z)+o(1)\equiv\frac{\pi}{2d}+o(1)\mod\frac{2\pi}{d}.
\end{equation}
Since $z\in \cS_j$, this implies
\begin{equation} \label{eq_argz}
\arg z\equiv\frac{\pi}{2d}+\frac{2\pi(j-1)}{d}+o(1)\mod 2\pi.
\end{equation}
Inserting \eqref{eq_argz} into \eqref{eq_zeros_im} yields
$$\im q(z)\equiv\arg(-c_j)+\lambda\left(\frac{\pi}{2}+2\pi(j-1)\right)+o(1)\mod 2\pi.$$
This completes the proof of the first part of Lemma \ref{lemma_zeros}. 

Let us now prove the second part. As before, we will give the proof  only for $k>0$, the proof for $k<0$ is analogous. Recall that $\vi_j$ is the branch of $q^{-1}$ that maps $\dC\setminus (\overline{\cD(0,R)}\cup [0,\infty))$ onto $\cS_j$. For small $\epsi>0$, let
$\cG_{j,k}$ be the interior of the set of all 
$$v\in \cH\left(\lambda,\,\frac{1+\varepsilon}{|c_j|},\,2|\lambda|\right)\setminus \cH\left(\lambda,\frac{1-\epsi}{|c_j|},\,2|\lambda|\right)$$
satisfying
$$|\im v-\im v_{j,k}|<\pi.$$
We will use the minimum principle to show that $g\circ \vi_j$ has a zero in $\cG_{j,k}$. For $v\in \cG_{j,k}$, we have $\re(v)=o(|v|)$, and hence
$$\arg(v)=\frac{\pi}{2}+o(1)$$
as $|v|\to\infty$.  Similar arguments as above and the definition of $v_{j,k}$ yield
\begin{equation}   \label{eq_argv}
\arg(\vi_j(v))\equiv\frac{\pi}{2d}+\frac{2\pi(j-1)}{d}+o(1)\equiv\frac{\arg(-c_j)-\im(v_{j,k})}{-d\lambda}+o(1)\mod \frac{2\pi}{d\lambda}.
\end{equation}
In particular, this is true for $v=v_{j,k}$. Also, since $v_{j,k}\in\Gamma(\lambda,1/|c_j|)$, we have
$$e^{-\re v_{j,k}} =\frac{1}{|c_j|}|v_{j,k}|^{-\lambda}=\frac{1}{|c_j|}|\vi_j(v_{j,k})|^{-d\lambda}(1+o(1)),$$
that is,
$$|\vi_j(v_{j,k})|^{-d\lambda}=|c_j|e^{-\re v_{j,k}}(1+o(1)).$$
Using Lemma \ref{lemma_asymptotics_g} and \eqref{eq_pq'}, we obtain
\begin{equation}  \label{eq_almostZero}
\begin{split}
(g\circ\vi_j)(v_{j,k})&=c_j+\vi_j(v_{j,k})^{-d\lambda}(1+o(1))\exp(v_{j,k})\\
&=c_j+|\vi_j(v_{j,k})|^{-d\lambda}\exp(-id\lambda\arg(\vi_j(v_{j,k})))(1+o(1))\exp(v_{j,k})\\
&=c_j+|c_j|\exp(-\re v_{j,k})\exp(i(\arg(-c_j)-\im(v_{j,k})))(1+o(1))\exp(v_{j,k})\\
&=c_j-c_j(1+o(1))=o(1).
\end{split}
\end{equation}
Next, we will show that $g\circ\vi_j$ is bounded below on $\partial \cG_{j,k}$.

If $v\in\Gamma(\lambda,(1+\epsi)/|c_j|)$, then
\begin{equation}
\begin{split}
|(g\circ\vi_j)(v)|&=|c_j+\vi_j(v)^{-d\lambda}(1+o(1))e^v|\ge||c_j|-|v|^{-\lambda}e^{\re v}(1+o(1))|\\
&=\left||c_j|-\frac{|c_j|}{1+\epsi}(1+o(1))\right|=\left|\frac{\epsi|c_j|}{1+\epsi}-o(1)\right|\ge\frac{\epsi|c_j|}{2},
\end{split}
\end{equation}
provided $\epsi$ is sufficiently small and $k$ is sufficiently large. An analogous estimate yields that if $v\in\Gamma(\lambda,(1-\epsi)/|c_j|)$, then
$$|(g\circ\vi_j)(v)|\ge\left|\frac{|c_j|}{1-\epsi}(1+o(1))-|c_j|\right|\ge\epsi|c_j|,$$
provided $\epsi$ is sufficiently small and $k$ is sufficiently large. If $\im(v)=\im(v_{j,k})\pm \pi$, then by \eqref{eq_argv},
$$\arg(\vi_j(v)^{-d\lambda}e^v)\equiv\arg(-c_j)\pm\pi+o(1)\equiv\arg(c_j)+o(1)\mod 2\pi.$$
Thus, for $v\in\overline{\cG_{j,k}}$ with $\im v=\im v_{j,k}\pm \pi$, we have
\begin{equation}
\begin{split}
|(g\circ\vi_j)(v)|&=\left|c_j+\vi_j(v)^{-d\lambda}(1+o(1))e^v\right|\\
&=\left||c_j|\exp(i\arg(c_j))+|v|^{-\lambda}|e^v|\exp(i\arg c_j+o(1))(1+o(1))\right|\\
&=\left||c_j|+|v|^{-\lambda}|e^v|(1+o(1))\right|\ge|c_j|.
\end{split}
\end{equation}
We obtain that if $k$ is sufficiently large, then
$$|(g\circ\vi_j)(v_{j,k})|=o(1)<\min_{v\in\partial \cG_{j,k}}|v|.$$
By the minimum principle, $g\circ \vi_j$ has a zero $w\in\cG_{j,k}$. The first part of the Lemma yields that $z:=\vi_j(w)$ satisfies \eqref{eq_zero}.
\end{proof}

\begin{cor}  \label{cor_zeros}
Let $j\in\{1,...,d\}$ and let $z\in \cS_j$ be a zero of $g$. Then 
$$\arg(z)=\begin{cases}\pi/(2d)+2\pi(j-1)/d+o(1)&\text{if }\im q(z)>0\\  -\pi/(2d)+2\pi j/d+o(1) &\text{if }\im q(z)<0\end{cases}$$
as $|z|\to\infty$. In particular,
$$\dist(z,\partial \cS_j)\ge\left(\frac{1}{d}+o(1)\right)|z|$$
as $|z|\to\infty$.
\end{cor}

\begin{proof}
The first part is stated in \eqref{eq_argz} in the case where $\im q(z)>0$, and follows from \eqref{eq_argq} with similar arguments in the case where $\im q(z)<0$. We obtain
$$\dist(z,\partial\cS_j)=\sin\left(\frac{\pi}{2d}+o(1)\right)|z|\ge\left(\frac{1}{d}+o(1)\right)|z|.\qedhere$$
\end{proof}


\section{The set \texorpdfstring{$q(\cF(f))$}{q(F(f))}: first part}  \label{sec_q(F)1}

For $j\in\{1,...,d\}$, let
$$\cF_j:=\cF(f)\cap \cS_j.$$
In Section \ref{sec_q(F)1}-\ref{sec_q(F)3}, we will investigate the location and size of $q(\cF_j)$ in three different subsets of $\dC$, using the sets $\cH(\mu, \alpha,\nu)$ introduced in Section \ref{sec_partition}. The first subset is $\cH(\lambda, 1/|c_j|, \nu)$,  the second one is $\cH(\lambda-1,\alpha_1/|c_j|,\nu)\setminus\cH(\lambda,\beta_1/|c_j|,\nu)$ for small $\alpha_1>0$ and large $\beta_1>0$, and the third set is $\{w:\,|\im w|\ge\nu\}\setminus\cH(\lambda-1,\beta_2/|c_j|,\nu)$ for large $\beta_2>0$. See Figure \ref{fig_partition_overview} for an illustration of these sets.
 
 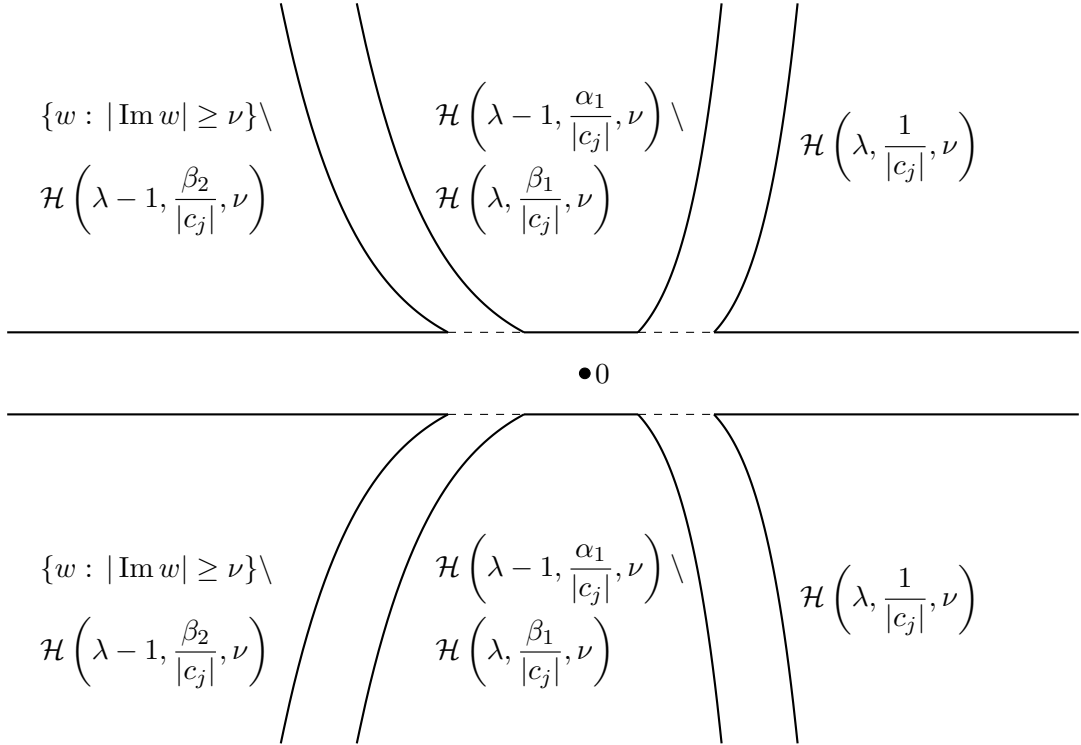
\begin{figure}[ht]
 \centering
 \begin{tikzpicture}
 \draw[thick, smooth, samples = 20, domain = -3.2:-1] plot(\x, {0.2*exp(-\x)});
 \draw[thick, smooth, samples = 20, domain = -4.2:-2] plot(\x, {0.2*exp(-(\x+1))});
 \draw[thick, smooth, samples = 20, domain = 0.5:1.6] plot(\x, {0.2*exp(2*\x)});
 \draw[thick, smooth, samples = 20, domain = 1.5:2.6] plot(\x, {0.2*exp(2*(\x-1))});

 \draw[thick] (-7.8, 0.544) -- (-2, 0.544);
 \draw[dashed] (-2, 0.544) -- (-1, 0.544);
 \draw[thick] (-1, 0.544) -- (0.5, 0.544);
 \draw[dashed] (0.5, 0.544) -- (1.5, 0.544);
 \draw[thick] (1.5, 0.544) -- (6.3, 0.544);
 
  \draw[thick] (-7.8, -0.544) -- (-2, -0.544);
 \draw[dashed] (-2, -0.544) -- (-1, -0.544);
 \draw[thick] (-1, -0.544) -- (0.5, -0.544);
 \draw[dashed] (0.5, -0.544) -- (1.5, -0.544);
 \draw[thick] (1.5, -0.544) -- (6.3, -0.544);
 
  \draw[thick, smooth, samples = 20, domain = -3.2:-1] plot(\x, {-0.2*exp(-\x)});
 \draw[thick, smooth, samples = 20, domain = -4.2:-2] plot(\x, {-0.2*exp(-(\x+1))});
 \draw[thick, smooth, samples = 20, domain = 0.5:1.6] plot(\x, {-0.2*exp(2*\x)});
 \draw[thick, smooth, samples = 20, domain = 1.5:2.6] plot(\x, {-0.2*exp(2*(\x-1))});
 
 \filldraw (-0.2,0) circle (2pt) node[right]{$0$};
 
 \draw (-7.5, 3.4) node[right] {$\{w:\,|\im w|\ge\nu\}\setminus$};
 \draw (-7.5, 2.3) node[right] {$\cH\left(\lambda-1,\dfrac{\beta_2}{|c_j|},\nu\right)$};
 \draw (-2.3, 3.4) node[right] {$\cH\left(\lambda-1,\dfrac{\alpha_1}{|c_j|},\nu\right)\setminus$};
 \draw (-2.3, 2.3) node[right] {$\cH\left(\lambda,\dfrac{\beta_1}{|c_j|},\nu\right)$};
 \draw (2.5,3) node[right] {$\cH\left(\lambda,\dfrac{1}{|c_j|},\nu\right)$};
 
 \draw (-7.5, -2.6) node[right] {$\{w:\,|\im w|\ge\nu\}\setminus$};
 \draw (-7.5, -3.7) node[right] {$\cH\left(\lambda-1,\dfrac{\beta_2}{|c_j|},\nu\right)$};
 \draw (-2.3, -2.6) node[right] {$\cH\left(\lambda-1,\dfrac{\alpha_1}{|c_j|},\nu\right)\setminus$};
 \draw (-2.3, -3.7) node[right] {$\cH\left(\lambda,\dfrac{\beta_1}{|c_j|},\nu\right)$};
 \draw (2.5,-3) node[right] {$\cH\left(\lambda,\dfrac{1}{|c_j|},\nu\right)$};
 
 \end{tikzpicture}
 \caption{An illustration of the sets $\cH(\lambda,1/|c_j|,\nu)$, $\cH(\lambda-1,\alpha_1/|c_j|,\nu)\setminus\cH(\lambda,\beta_1/|c_j|,\nu)$ and $\{w:\,|\im w|\ge\nu\}\setminus\cH(\lambda-1,\beta_2/|c_j|,\nu)$ in the case where $\lambda>0$.}
 \label{fig_partition_overview}
 \end{figure}

 In this section, we investigate the location and size of $q(\cF_j)$ in $\cH(\lambda, 1/|c_j|, \nu)$ for $j\in\{1,...,d\}$ and large $\nu>0$. Recall that the branch $\vi_j$ of $q^{-1}$ maps $\cH(\lambda, 1/|c_j|, \nu)$ to a subset of $\cS_j$.
 
 \begin{lemma}    \label{lemma_invariance_Sj}
Let  $j\in\{1,...,d\}$. There exists $\nu_0>0$ such that
$$(f\circ\vi_j)(\cH(\lambda,1/|c_j|,\nu_0))\subset \cS_j.$$
In particular, if $(q\circ f\circ\vi_j)^k(w)\in\cH(\lambda,1/|c_j|,\nu_0)$ for all $k\in\{0,...,n-1\}$, then $(f^n\circ\vi_j)(w)\in\cS_j$ and $(q\circ f\circ\vi_j)^n(w)=(q\circ f^n\circ\vi_j)(w)$.
\end{lemma}

\begin{proof}
Let $w\in \cH(\lambda,1/|c_j|,\nu_0).$ By Corollary \ref{cor_asymptotics_f}, \eqref{eq_pq'vi} and Remark \ref{remark_gamma}, 
\begin{equation}
\begin{split}
&|(f\circ\vi_j)(w)-\vi_j(w)|\\
&\le\frac{1}{|q'(\vi_j(w))|}\left(1+O\left(\frac{1}{|w|}\right)+\frac{|q'(\vi_j(w))c_je^{-w}|}{|p(\vi_j(w))|}\right)\\
&=\frac{1}{|q'(\vi_j(w))|}\left(1+O\left(\frac{1}{|w|}\right)+|w|^{\lambda}|c_je^{-w}|\left(1+O\left(\frac{1}{|w|^{1/d}}\right)\right)\right)\\
&\le\frac{3}{|q'(\vi_j(w))|}=3|\vi_j'(w)|
\end{split}
\end{equation}
if $|w|$ is sufficiently large. For $\nu_0\ge12+R$, with $R$ as in Section \ref{sec_changeOfVar},  we obtain
$$f(\vi_j(w))\in \cD(\vi_j(w),3|\vi_j'(w)|)\subset \cD\left(\vi_j(w),\frac{\nu_0-R}{4}|\vi_j'(w)|\right).$$
On the other hand, by Koebe's $1/4$-theorem, 
$$\cS_j\supset \vi_j\left(\cD\left(w,\nu_0-R\right)\right)\supset \cD\left(\vi_j(w),\frac{\nu_0-R}{4}|\vi_j'(w)|\right),$$
whence the claim follows.
\end{proof}
 
Next, we derive an asymptotic expression for
$$h_j(w)=(q\circ f \circ\vi_j)(w)$$
in $\cH(\lambda,2/|c_j|,\nu_1)$ for large $\nu_1>0$.

 \begin{lemma}   \label{lemma_asymptotics_rechts}
Let $j\in\{1,...,d\}$. There exists $\nu_1>0$ such that 
\begin{equation}  \label{eq_asymptotics_rechts}
h_j(w)=w-1+\frac{2m+1-d}{2d}\cdot\frac{1}{w}+O\left(\frac{1}{|w|^{1+1/d}}\right)-c_je^{-w}\vi_j(w)^{d\lambda}\left(1+O\left(\frac{1}{|w|^{1/d}}\right)\right)
\end{equation}
as $w\to\infty$ in $\cH(\lambda,2/|c_j|,\nu_1)$.
\end{lemma}

\begin{remark}
In fact, for any $\alpha>0$, there exists $\nu>0$ such that $h_j$ has an asymptotic expression of the form \eqref{eq_asymptotics_rechts} in $\cH(\lambda,\alpha/|c_j|,\nu)$. We will need that $\alpha>1$ so that $\cH(\lambda,\alpha/|c_j|,\nu)\supset\cH(\lambda,1/|c_j|,\nu)$.
\end{remark}

\begin{proof}[Proof of Lemma \ref{lemma_asymptotics_rechts}]
By Corollary \ref{cor_asymptotics_f}, we have
$$f(\vi_j(w))=\vi_j(w)-\frac{\eta(w)}{q'(\vi_j(w))}$$
where 
$$\eta(w)=1+\frac{\lambda}{w}+O\left(\frac{1}{|w|^{1+1/d}}\right)+c_je^{-w}\frac{q'(\vi_j(w))}{p(\vi_j(w))}$$
as $|w|\to\infty$. Note that $\eta$ is bounded in $\cH(\lambda,2/|c_j|, \nu_1)$. Taylor expansion of $q$ around $\vi_j(w)$ yields
\begin{align}  
h_j(w)&=q(f(\vi_j(w)))=\sum_{k=0}^d\frac{1}{k!}q^{(k)}(\vi_j(w))(f(\vi_j(w))-\vi_j(w))^k\\
&=\sum_{k=0}^d\frac{(-1)^k}{k!}\frac{q^{(k)}(\vi_j(w))}{q'(\vi_j(w))^k}\eta(w)^k\\
&=w-\eta(w)+\frac{1}{2}\frac{q''(\vi_j(w))}{q'(\vi_j(w))^2}\eta(w)^2+\sum_{k=3}^d\frac{(-1)^k}{k!}O\left(\frac{1}{w^{k-1}}\right)\eta(w)^k\\
&=w-\eta(w)+\frac{1}{2}\frac{q''(\vi_j(w))}{q'(\vi_j(w))^2}\eta(w)^2+O\left(\frac{1}{w^2}\right)   \label{eq_asymprechts_1}
\end{align}
as $w\to\infty$ in $\cH(\lambda,2/|c_j|,\nu_1)$. Using that $\lambda=(d-1-m)/d$, we have
\begin{equation}  \label{eq_asymprechts_2}
-\eta(w)=-1+\frac{m+1-d}{d}\cdot\frac{1}{w}+O\left(\frac{1}{|w|^{1+1/d}}\right)-c_je^{-w}\vi_j(w)^{d-1-m}\left(1+O\left(\frac{1}{|w|^{1/d}}\right)\right).
\end{equation}
Moreover,
$$\frac{q''(\vi_j(w))}{q'(\vi_j(w))^2}=\frac{d-1}{d}\cdot\frac{1}{w}\left(1+O\left(\frac{1}{|w|^{1/d}}\right)\right)$$
and
\begin{equation}
\begin{split}
\eta(w)^2&=\left(1+O\left(\frac{1}{w}\right)+c_je^{-w}\vi_j(w)^{d\lambda}\left(1+O\left(\frac{1}{|w|^{1/d}}\right)\right)\right)^2\\
&=1+O\left(\frac{1}{w}\right)+c_je^{-w}\vi_j(w)^{d\lambda}\cdot O(1).
\end{split}
\end{equation}
Hence,
\begin{equation}   \label{eq_asymprechts_3}
\frac{1}{2}\frac{q''(\vi_j(w))}{q'(\vi_j(w))^2}\eta(w)^2=\frac{d-1}{2d}\frac{1}{w}+O\left(\frac{1}{|w|^{1+1/d}}\right)+c_je^{-w}\vi_j(w)^{d\lambda}O\left(\frac{1}{w}\right).
\end{equation}

Combining \eqref{eq_asymprechts_1}, \eqref{eq_asymprechts_2} and \eqref{eq_asymprechts_3} yields the desired conclusion.
\end{proof}

For the derivative of $h_j$, we obtain the following.

\begin{lemma}   \label{lemma_asymptotics_derivative_rechts}
Let $j\in\{1,...,d\}$. There exists $\nu_2>0$ such that
$$h_j'(w)=1+O\left(\frac{1}{|w|^{1+1/d}}\right)+c_je^{-w}\vi_j(w)^{d\lambda}\left(1+O\left(\frac{1}{|w|^{1/d}}\right)\right)$$
as $w\to\infty$ in $\cH(\lambda,1/|c_j|, \nu_2)$.
\end{lemma}

\begin{proof}
Suppose $\nu_2\ge\nu_1+1$. By Lemma \ref{lemma_asymptotics_rechts}, there are holomorphic functions, $a_1, a_2$, satisfying $a_1(w)=O(1/|w|^{1+1/d})$ and $a_2(w)=O(1/|w|^{1/d})$ as $w\to\infty$ such that
$$h_j(w)=w-1+\frac{2m+1-d}{2d}\cdot\frac{1}{w}+a_1(w)-c_je^{-w}\vi_j(w)^{d\lambda}\left(1+a_2(w)\right)$$
for $w\in\cH(\lambda, 2/|c_j|, \nu_2-1)$. By Lemma \ref{lemma_gamma} and Cauchy's inequality, we have $a_1'(w)=O(1/|w|^{1+1/d})$ and $a_2'(w)=O(1/|w|^{1/d})$ as $w\to\infty $ in $\cH(\lambda,1/|c_j|, \nu_2)$. Also,
\begin{equation}
\begin{split}
\frac{\diff}{\diff w}e^{-w}\vi_j(w)^{d\lambda}&=-e^{-w}\vi_j(w)^{d\lambda}\left(1-\frac{d\lambda}{\vi_j(w)}\vi_j'(w)\right)\\
&=-e^{-w}\vi_j(w)^{d\lambda}\left(1+O\left(\frac{1}{w}\right)\right).
\end{split}
\end{equation}
Thus,
$$\frac{\diff}{\diff w} \left(c_je^{-w}\vi_j(w)^{d\lambda}(1+a_2(w))\right)=-c_je^{-w}\vi_j(w)^{d\lambda}\left(1+O\left(\frac{1}{|w|^{1/d}}\right)\right).$$
We obtain
\begin{equation}
\begin{split}
h_j'(w)
&=1-\frac{2m+1-d}{2d}\frac{1}{w^2}+a_1'(w)+c_je^{-w}\vi_j(w)^{d\lambda}\left(1+O\left(\frac{1}{|w|^{1/d}}\right)\right)\\
&=1+O\left(\frac{1}{|w|^{1+1/d}}\right)+c_je^{-w}\vi_j(w)^{d\lambda}\left(1+O\left(\frac{1}{|w|^{1/d}}\right)\right)
\end{split}
\end{equation}
as $w\to\infty$ in $\cH(\lambda,1/|c_j|,\nu_2)$.
\end{proof}

We will now proceed as follows. Recall that if $z_0\in\cS_j$ is a superattracting fixed point of $f$, then $q(z_0)$ lies close to the curve $\Gamma(\lambda,1/|c_j|)$. Also, every horizontal strip of width $2\pi+\epsi$ with $\epsi>0$ that is sufficiently far from the real axis contains such an image of a superattracting fixed point. We will show that if $z_0$ is a superattracting fixed point of $f$ and $\cA^*(z_0)$ is its immediate basin of attraction, then $q(\cA^*(z_0))$ contains a disk of a fixed radius around $q(z_0)$. We then consider preimages of this disk under iterates of $h_j=q\circ f\circ\vi_j$. The function $h_j$ is not locally invertible at $q(z_0)$, but if $\alpha$ is slightly smaller than $1$, then $h_j$ has a local inverse function, $\psi_j$, defined in $\cH(\lambda,\alpha/|c_j|,\nu)$. If $\alpha$ is sufficiently close to $1$, then $\cH(\lambda,\alpha/|c_j|,\nu)$ intersects the disks contained in $q(\cA^*(z_0))$. We then show that the images of this intersection under $\psi_j$ have a certain size and are more or less evenly distributed in $\cH(\lambda,1/|c_j|,\nu)$. The idea here is that  if $w\in\cH(\lambda,1/|c_j|,\nu)$ is sufficiently far from the boundary, then Lemma \ref{lemma_asymptotics_rechts} yields $h_j(w)\approx w-1$, and hence $\psi_j(w)\approx w+1$. See Figure \ref{fig_q(F)_rechts} for an illustration of the abovementioned approach.

\begin{figure}[ht]  
\centering
\begin{tikzpicture}[scale=1.5]
\draw[thick, dashed, smooth, samples = 20, domain = 0.5:1.5] plot (\x, {sqrt(0.1*exp(4*\x)-\x*\x)});
\draw[thick, smooth, samples = 20, domain = 0.7:1.7] plot (\x, {sqrt(0.1*exp(4*(\x-0.2))-(\x-0.2)*(\x-0.2))});

\draw[thick] (1.35, 5.01) circle (0.6);
\filldraw (1.35, 5.01) circle (2pt);
\filldraw[color = black, fill = black!35, thick] (1.77, 5.01) circle (0.15);
\draw[->, thick] (2.05, 5.01) -- (3.25, 5) node[midway, above]{$\psi_{j}$};
\filldraw[color = black, fill = black!35, thick, rounded corners = 1pt] (3.75, 5) .. controls (3.71, 5.11) and (3.6, 5.2) .. (3.49, 5.11)
	.. controls (3.45, 5) .. (3.44, 4.89)
	.. controls (3.6, 4.9) .. (3.76, 4.89)
	-- cycle;
\draw[->, thick] (3.95,5) -- (5.05,5.2) node[midway, above]{$\psi_{j}$};
\filldraw[color = black, fill = black!35, thick, rounded corners = 1pt] (5.55, 5.2) .. controls (5.51, 5.31) and (5.4, 5.4) .. (5.29, 5.31)
	.. controls (5.25, 5.2) .. (5.24, 5.04)
	.. controls (5.4, 5.15) .. (5.56, 5.04)
	-- cycle;
\draw[->, thick] (5.75, 5.2) -- (7.05, 5.25) node[midway, above]{$\psi_{j}$};
\filldraw[color = black, fill = black!35, thick, rounded corners = 1pt] (7.6, 5.3) .. controls (7.56, 5.41) and (7.4, 5.5) .. (7.29, 5.41)
	.. controls (7.25, 5.3) .. (7.24, 5.14)
	.. controls (7.4, 5.25) .. (7.61, 5.14)
	-- cycle;

\draw[thick] (1,1.89) circle (0.6);
\filldraw (1, 1.89) circle (2pt);
\filldraw[color = black, fill = black!35, thick] (1.4, 1.89) circle (0.15);
\draw[->, thick] (1.7, 1.89) -- (2.7, 1.9) node[midway, above]{$\psi_{j}$};
\filldraw[color = black, fill = black!35, thick, rounded corners = 1pt]  (3.15, 1.9) .. controls (3.11, 2.01) and (3, 2) .. (2.89, 2.01) 	.. controls (2.9, 1.9) .. (2.89, 1.79)
	.. controls (3, 1.75) .. (3.16, 1.84)
	-- cycle;
\draw[->, thick] (3.35, 1.9) -- (4.5, 1.6) node[midway, above]{$\psi_{j}$};
\filldraw[color = black, fill = black!35, thick, rounded corners = 1pt]  (5, 1.6) .. controls (4.91, 1.71) and (4.8, 1.7) .. (4.69,1.71) 	.. controls (4.7, 1.6) .. (4.69, 1.49)
	.. controls (4.8, 1.45) .. (4.91, 1.59)
	-- cycle;
\draw[->, thick] (5.2, 1.6) -- (6.4, 1.7) node[midway, above]{$\psi_{j}$};
\filldraw[color = black, fill = black!35, thick, rounded corners = 1pt]  (6.9, 1.8) .. controls (6.81, 1.81) and (6.7, 1.8) .. (6.59,1.81) 	.. controls (6.6, 1.7) .. (6.59, 1.59)
	.. controls (6.7, 1.55) .. (6.81, 1.69)
	-- cycle;
	
\end{tikzpicture}
\caption{The images of superattracting fixed points of $f$ under $q$ lie close to the dashed line. The white disks around them are contained in the images of the basins of attraction of the superattracting fixed points. To the right of the solid line, the inverse $\psi_j$ of $h_j$ is defined. The grey disks lie in the intersection of the images of the basins of attraction under $q$ and the domain of definition of $\psi_j$, and their images under iteration of $\psi_j$ are contained in $q(\cF_j)$.}
\label{fig_q(F)_rechts}
\end{figure}
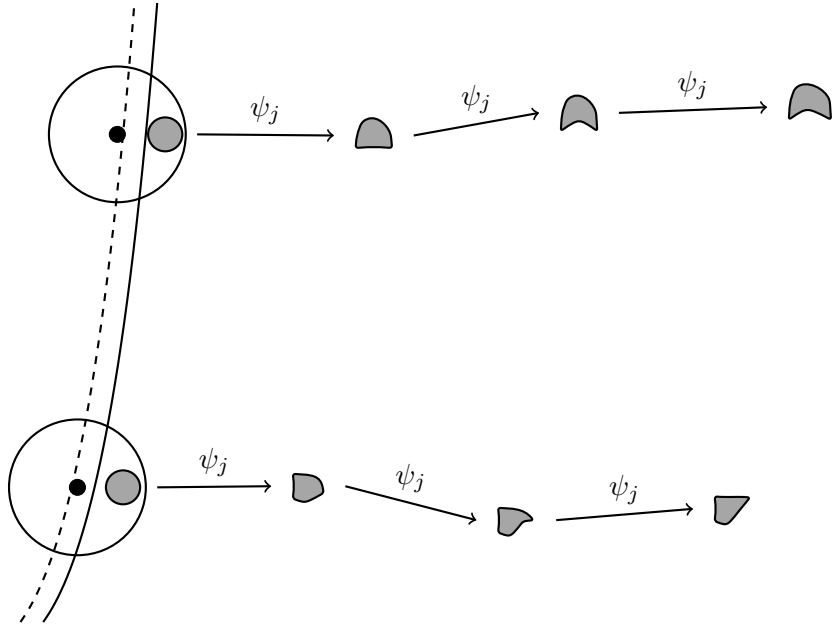

\begin{lemma}     \label{lemma_disk_attr_basin}
If $z_0$ is a zero of $g$ but not a zero of $g'$ and $|z_0|$ is sufficiently large, then
$$\cA^*(z_0)\supset \cD\left(z_0,\frac{1}{3d|z_0|^{d-1}}\right).$$
\end{lemma}

For the proof, we require the following theorem which essentially says that under suitable assumptions the solution of B\"ottcher's functional equation in a neighbourhood of a superattracting fixed point extends to a conformal map defined in the entire immediate basin of attraction.

\begin{thm}    \label{thm_conjugation_attr_basin}
Let $h$ be a meromorphic function, and let $z_0$ be a superattracting fixed point of multiplicity $k$ of $h$. Suppose that $\cA^*(z_0)$ contains no critical point other than $z_0$ and no asymptotic value of $h$. Then there is a conformal map $\Phi:\cD(0,1)\to \cA^*(z_0)$ satisfying $\Phi(0)=z_0$ and
$$h(\Phi(z))=\Phi(z^k)$$
for all $z\in\cD(0,1)$.
\end{thm}

A proof of this theorem can be found, for example, in \cite[p. 65, Theorem 4]{steinmetz93}. There, the result is stated for rational functions, but the proof also works for meromorphic functions without asymptotic values in $\cA^*(z_0)$.

\begin{proof}[Proof of Lemma \ref{lemma_disk_attr_basin}]
Let $z_0$ be a zero of $g$ that is not a zero of $g'$, and assume that none of the finitely many zeros of $g''$ lies in $\cA^*(z_0)$. Then $z_0$ is a superattracting fixed point of $f$, and there are no other critical points of $f$ in $\cA^*(z_0)$. Also, 
$$f''(z)=\frac{g'(z)^2g''(z)+g(z)g'(z)g'''(z)-2g(z)g''(z)^2}{g'(z)^3},$$
and hence
$$f''(z_0)=\frac{g''(z_0)}{g'(z_0)}\ne0.$$
By Theorem \ref{thm_conjugation_attr_basin}, there is a conformal map $\Phi:\cD(0,1)\to \cA^*(z_0)$ satisfying $f(\Phi(z))=\Phi(z^2)$ and $\Phi(0)=z_0.$ 
Differentiating the equation $f(\Phi(z))=\Phi(z^2)$ twice yields
$$f''(\Phi(z))\Phi'(z)^2+f'(\Phi(z))\Phi''(z)=2\Phi'(z^2)+4z^2\Phi''(z^2).$$
For $z=0$, we obtain
$$f''(z_0)\Phi'(0)^2=2\Phi'(0)$$
and hence
$$|\Phi'(0)|=\frac{2}{|f''(z_0)|}.$$
We have
\begin{equation}
\begin{split}
f''(z_0)&=\frac{g''(z_0)}{g'(z_0)}=\frac{(p(z_0)q'(z_0)+p'(z_0))e^{q(z_0)}}{p(z_0)e^{q(z_0)}}\\
&=q'(z_0)+\frac{p'(z_0)}{p(z_0)}=dz_0^{d-1}\left(1+O\left(\frac{1}{z_0}\right)\right).
\end{split}
\end{equation}
Hence, by Koebe's $1/4$-Theorem,
$$\cA^*(z_0)=\Phi(\cD(0,1))\supset \cD\left(z_0,\frac{1}{4}|\Phi'(0)|\right)=\cD\left(z_0,\frac{1}{2|f''(z_0)|}\right)\supset \cD\left(z_0,\frac{1}{3d|z_0|^{d-1}}\right)$$
if $|z_0|$ is sufficiently large.
\end{proof}

\begin{cor}   \label{cor_disk_attr_basin}
Let $z_0\in\dC$ be a zero of $g$ that is not a zero of $g'$. If $|z_0|$ is sufficiently large, then
$$q(\cA^*(z_0))\supset \cD\left(q(z_0),\frac{1}{13}\right).$$
\end{cor}

\begin{proof}
If $|z_0|$ is sufficiently large, then by Lemma \ref{lemma_disk_attr_basin},
$$\cA^*(z_0)\supset \cD\left(z_0,\frac{1}{3d|z_0|^{d-1}}\right),$$
and $q$ is injective in this disk. Koebe's $1/4$-Theorem yields
$$q(\cA^*(z_0))\supset q\left(\cD\left(z_0,\frac{1}{3d|z_0|^{d-1}}\right)\right)\supset \cD\left(q(z_0),\frac{|q'(z_0)|}{12d|z_0|^{d-1}}\right).$$
Since $q'(z)=dz^{d-1}(1+O(1/z))$ as $z\to\infty$, the claim follows.
\end{proof}

The next Lemma deals with preimages under $h_j$.

\begin{lemma}  \label{lemma_preimages}
Let $\alpha\in(0,1),\,\epsi\in(0,1-\alpha)$ and $j\in\{1,...,d\}$. There exists $\nu_3>0$ such that for each $w_0\in \cH(\lambda,\alpha/|c_j|,\nu_3)$, there is a unique $w\in \cH(\lambda,\alpha/|c_j|,\nu_3-1)$ with $h_j(w)=w_0.$ More precisely,
$w\in \cD(w_0+1, \alpha+\epsi)$.
\end{lemma}

\begin{proof}
Let $w\in \cH(\lambda,\alpha/|c_j|,\nu_3-1).$ By Lemma \ref{lemma_asymptotics_rechts} and Remark \ref{remark_gamma},
\begin{equation}
\begin{split}
|h_j(w)-(w-1)|&\le O\left(\frac{1}{|w|}\right)+|c_je^{-w}||w|^{\lambda}\left(1+O\left(\frac{1}{|w|^{1/d}}\right)\right)\\
&\le O\left(\frac{1}{|w|}\right)+\alpha\left(1+O\left(\frac{1}{|w|^{1/d}}\right)\right)<\alpha+\epsi,
\end{split}
\end{equation}
provided $\nu_3$ and hence $|w|$ is sufficiently large. If $h_j(w)=w_0$, we obtain
$$|w-(w_0+1)|=|w_0-(w-1)|=|h_j(w)-(w-1)|<\alpha+\epsi,$$
that is, $w\in \cD(w_0+1,\alpha+\epsi)$. On the other hand, Lemma \ref{lemma_gamma} yields that  
$$\overline{\cD(w_0+1,\alpha+\epsi)}\subset \cH(\lambda,\alpha/|c_j|,\nu_3-1)$$
 if $\nu_3$ is sufficiently large. Thus, for $w\in\partial \cD(w_0+1,\alpha+\epsi)$,
$$|(h_j(w)-w_0)-(w-1-w_0)|=|h_j(w)-(w-1)|<\alpha+\epsi=|w-1-w_0|.$$
By Rouch\'e's theorem, there is a unique $w\in \cD(w_0+1,\alpha+\epsi)$ satisfying $h_j(w)=w_0$.
\end{proof}

By Lemma \ref{lemma_preimages}, there is a subset $\cH_j\subset \cH(\lambda,\alpha/|c_j|,\nu_3-1)$ such that $h_j$ maps $\cH_j$ conformally onto $\cH(\lambda,\alpha/|c_j|,\nu_3)$. Let $\psi_j:\cH(\lambda,\alpha/|c_j|,\nu_3)\to \cH_j$ be the corresponding inverse function. The next Lemma yields that if $|\im w|$ is sufficiently large, then all iterates $\psi_j^n(w)$ are defined and tend to $\infty$ as $n\to\infty$ in a horizontal strip whose width is bounded independent of $w$.

\begin{lemma}   \label{lemma_preimages_seq}
Let $\alpha\in(0,1)$, $\epsi\in(0,1-\alpha)$ and $j\in\{1,...,d\}$. Then there exist $\nu_4>0$ and $C>0$ such that $\psi_j^n(w)$ is defined for all $w\in \cH(\lambda,\alpha/|c_j|,\nu_4)$ and all $n\in\dN$, and satisfies
\begin{enumerate}[(i)]
\item $\re\psi_j^n(w)\ge\re w+n(1-\alpha-\epsi);$
\item $|\psi_j^n(w)|\ge\max\{n,|w|\}\cdot\dfrac{1-\alpha-\epsi}{4}$;
\item $|\im\psi_j^n(w)-\im w|\le C$; 
\item $e^{-\re\psi_j^n(w)}|\psi_j^n(w)|^{\lambda}=O(e^{-n(1-\alpha-\epsi)/2})$.
\end{enumerate}
\end{lemma}

For the proof, we require the following Lemma.

 \begin{lemma}  \label{lemma_series_bound}
For all $n_0\in\dN$, we have
 $$\sum_{n=n_0}^\infty\frac{1}{k^2}\le\frac{2}{n_0}.$$
 \end{lemma}

\begin{proof}
We have
$$\sum_{n=n_0}^\infty\frac{1}{k^2}\le\frac{1}{n_0^2}+\sum_{k=n_0+1}^\infty\int_{k-1}^k\frac{1}{t^2} \,dt =\frac{1}{n_0^2}+\int_{n_0}^\infty\frac{1}{t^2} \,dt=\frac{1}{n_0^2}+\frac{1}{n_0}\le\frac{2}{n_0}.\qedhere$$
\end{proof}

\begin{proof}[Proof of Lemma \ref{lemma_preimages_seq}]
Let 
$$\delta:=1-\alpha-\epsi.$$
First note that if $\psi_j^n(w)$ is defined, then Lemma \ref{lemma_preimages} yields that $\psi_j^k(w)\in \cD(\psi_j^{k-1}(w)+1, \alpha+\epsi)$ for all $k\in\{1,...,n\}$, and hence 
$$\re \psi_j^n(w)\ge\re w+n\delta.$$
So $\psi_j^n(w)$ satisfies (i).
Also, if $n\le|w|/2$, then
\begin{equation}
|\psi_j^n(w)|\ge|w|-n(1+\alpha+\epsi)\ge|w|-\frac{|w|}{2}(1+\alpha+\epsi)=|w|\frac{\delta}{2}\ge n\delta.
\end{equation}
If $n>|w|/2$, then
\begin{equation}
|\psi_j^n(w)|\ge\re\psi_j^n(w)\ge\re w+n\delta\ge\lambda\log|w|-\log\frac{\alpha}{|c_j|}+n\delta\ge\frac{n\delta}{2}\ge\frac{|w|\delta}{4},
\end{equation}
provided $|w|$ and hence also $n$ is sufficiently large.
In particular, $\psi_j^n(w)$ satisfies (ii).

Let 
$$n_w:=\lfloor |w|\rfloor.$$
We will show by induction that if $w\in \cH(\lambda,\alpha/|c_j|,\nu_4)$ for sufficiently large $\nu_4>0$, then $\psi_j^n(w)$ is defined for all $n\in\dN$ and
\begin{equation}  \label{eq_preimages_induction}
|\im\psi_j^n(w)-\im w|\le C'\left(\min\left\{\frac{n}{|w|},1\right\}+n_w\sum_{k=n_w}^n\frac{1}{k^2}+\sum_{k=1}^n\frac{1}{k^{1+1/d}}+\sum_{k=1}^ne^{-k\delta/2}\right),
\end{equation}
where $C'$ does not depend on $w$ or $n$. Note that by Lemma \ref{lemma_series_bound},
\begin{equation}
\begin{split}
&C'\left(\min\left\{\frac{n}{|w|},1\right\}+n_w\sum_{k=n_w}^n\frac{1}{k^2}+\sum_{k=1}^n\frac{1}{k^{1+1/d}}+\sum_{k=1}^ne^{-k\delta/2}\right)\\
\le& C'\left(3+\sum_{k=1}^\infty\frac{1}{k^{1+1/d}}+\sum_{k=1}^\infty e^{-k\delta/2}\right)=:C<\infty.
\end{split}
\end{equation}

So \eqref{eq_preimages_induction} implies (iii). Clearly, \eqref{eq_preimages_induction} is true for $n=0$. Now suppose that \eqref{eq_preimages_induction} holds with $n$ replaced by $n-1$. By Lemma \ref{lemma_preimages}, $\psi_j^n(w)$ is defined if and only if $|\im\psi_j^{n-1}(w)|>\nu_3$. This is satisfied if $|\im w|>\nu_3+C$. By Lemma \ref{lemma_asymptotics_rechts}, 
\begin{equation}
\begin{split}
&|\im\psi_j^n(w)-\im\psi_j^{n-1}(w)|=|\im\psi_j^n(w)-\im h_j(\psi_j^n(w))|\\
&\le\left|\frac{2m+1-d}{2d}\im\left(\frac{1}{\psi_j^n(w)}\right)\right|+O\left(\frac{1}{|\psi_j^n(w)|^{1+1/d}}\right)+2|c_j|e^{-\re\psi_j^n(w)}|\psi_j^n(w)|^{\lambda},
\end{split}
\end{equation}
provided $|w|$ is sufficiently large. By (ii),
\begin{equation}
\frac{1}{|\psi_j^n(w)|^{1+1/d}}=O\left(\frac{1}{n^{1+1/d}}\right).
\end{equation}
If $n\le|w|$, then we estimate the first summand by
$$\left|\im\left(\frac{1}{\psi_j^n(w)}\right)\right|\le\frac{1}{|\psi_j^n(w)|}=O\left(\frac{1}{|w|}\right).$$
If $n>|w|$, then by Lemma \ref{lemma_preimages},  (ii) and the induction hypothesis,
\begin{equation}
\begin{split}
\left|\im\left(\frac{1}{\psi_j^n(w)}\right)\right|&=\frac{|\im \psi_j^n(w)|}{|\psi_j^n(w)|^2}\le\frac{|\im \psi_j^{n-1}(w)|+\alpha+\epsi}{|\psi_j^n(w)|^2}\\
&\le\frac{16(|\im w|+C+\alpha+\epsi)}{\delta^2n^2}\le\frac{17|w|}{\delta^2n^2}=n_w\cdot O\left(\frac{1}{n^2}\right),
\end{split}
\end{equation}
provided $|w|$ is sufficiently large. 

Moreover, if $\lambda\ge0$, then by (i), Lemma \ref{lemma_preimages} and Remark \ref{remark_gamma},
\begin{equation}
\begin{split}
|c_j|e^{-\re \psi_j^n(w)}|\psi_j^n(w)|^{\lambda}&\le|c_j|e^{-\re w}(|w|+n(1+\alpha+\epsi))^{\lambda}e^{-n\delta}\\
&\le \alpha|w|^{-\lambda}(|w|+n(1+\alpha+\epsi))^{\lambda}e^{-n\delta}\\
&=\alpha\left(1+\frac{n}{|w|}(1+\alpha+\epsi)\right)^{\lambda}e^{-n\delta}\\
&\le\alpha(1+n(1+\alpha+\epsi))^{\lambda}e^{-n\delta}=O(e^{-n\delta/2}),
\end{split}
\end{equation}
provided $|w|\ge1$. 
If $\lambda<0$, then by (i), (ii) and Remark \ref{remark_gamma},
\begin{equation}
|c_j|e^{-\re \psi_j^n(w)}|\psi_j^n(w)|^{\lambda}\le\left(\frac{\delta}{4}\right)^{\lambda}|w|^{\lambda}|c_j|e^{-\re w}e^{-n\delta}\le\left(\frac{\delta}{4}\right)^{\lambda}\alpha e^{-n\delta}.
\end{equation}
In particular, (iv) is satisfied. Also, if $n\le|w|$, then
$$\im\psi_j^n(w)-\im\psi_j^{n-1}(w)=O\left(\frac{1}{|w|}\right)+O\left(\frac{1}{n^{1+1/d}}\right)+O\left(e^{-n\delta/2}\right),$$
and if $n>|w|$, then
$$\im\psi_j^n(w)-\im\psi_j^{n-1}(w)=n_wO\left(\frac{1}{n^2}\right)+O\left(\frac{1}{n^{1+1/d}}\right)+O(e^{-n\delta/2}).$$
Thus, $\psi_j^n(w)$ satisfies \eqref{eq_preimages_induction}, and hence also (iii).
\end{proof}

We now estimate the derivative of $\psi_j^n$.

\begin{lemma}  \label{lemma_lower_bound_derivative}
Let $\alpha\in(0,1)$ and $j\in\{1,...,d\}$. There are $\nu_5>0$ and $B>0$ such that 
$$|(\psi_j^n)'(w)|\ge B$$
for all $w\in \cH(\lambda,\alpha/|c_j|, \nu_5)$ and all $n\in\dN$.
\end{lemma}

\begin{proof}
We have
$$(\psi_j^n)'=\prod_{k=0}^{n-1}\psi_j'\circ\psi_j^k=\frac{1}{\prod_{k=0}^{n-1}h_j'\circ\psi_j^{k+1}}=\frac{1}{\prod_{k=1}^nh_j'\circ\psi_j^k}.$$
By Lemma \ref{lemma_asymptotics_derivative_rechts} and Lemma \ref{lemma_preimages_seq},
\begin{equation}
\begin{split}
|h_j'(\psi_j^{k}(w))|
&\le1+O\left(\frac{1}{|\psi_j^{k}(w)|^{1+1/d}}\right)+|c_j|e^{-\re\psi_j^{k}(w)}|\psi_j^{k}(w)|^{\lambda}\left(1+O\left(\frac{1}{|\psi_j^{k}(w)|^{1/d}}\right)\right)\\
&\le1+O\left(\frac{1}{k^{1+1/d}}\right)+O\left(e^{-k(1-\alpha-\epsi)/2}\right).
\end{split}
\end{equation}
Since the infinite product $\prod_{k=1}^\infty(1+O(1/k^{1+1/d})+O(e^{-k(1-\alpha-\epsi)/2}))$ converges, we obtain the desired conclusion.
\end{proof}

Recall that $\cF_j=\cF(f)\cap\cS_j$.

\begin{lemma}  \label{lemma_disk_q(F)}
For $j\in\{1,...,d\}$ and $k\in\dZ$, let $v_{j,k}$ be as in Lemma \ref{lemma_zeros} and $w_{j,k}:=v_{j,k}+1/26$. There is $\teta>0$ such that if $|k|$ is sufficiently large, then $\cD(\psi_j^n(w_{j,k}),\teta)\subset q(\cF_j)$ for all $n\in\dN$.
\end{lemma}

\begin{remark}
For sufficiently large $|k|$, the point $v_{j,k}$ is close to $q(z_0)$ for some attracting fixed point $z_0$ of $f$. The function $\psi_j$ is not defined in $q(z_0)$ and $v_{j,k}$. Therefore, we introduce the point $w_{j,k}$ which is in the domain of definition of $\psi_j$ for large $|k|$ and also lies in $q(A^*(z_0))$.
 \end{remark}

\begin{proof}[Proof of Lemma \ref{lemma_disk_q(F)}]
By Lemma \ref{lemma_zeros}, there is a zero $z_0$ of $g$ satisfying $q(z_0)=v_{j,k}+o(1)$. Thus, $w_{j,k}=q(z_0)+1/26+o(1)$.  If $|k|$ is sufficiently large, we obtain 
$$\cD\left(w_{j,k},\frac{1}{27}\right)\subset \cD\left(q(z_0)+\frac{1}{26},\frac{1}{26}\right)\subset\cD\left(q(z_0),\frac{1}{13}\right).$$
By Corollary \ref{cor_disk_attr_basin}, this yields
$$\cD\left(w_{j,k},\frac{1}{27}\right)\subset q(\cA^*(z_0)).$$
Let $\nu>0$ be large and $\exp(-1/2(1/26-1/27))<\alpha<1.$ Then
\begin{equation}  \label{eq_curvedist}
2\log\frac{1}{\alpha}<\frac{1}{26}-\frac{1}{27}.
\end{equation}
Since $v_{j,k}\in\Gamma(\lambda,1/|c_j|)$, by \eqref{eq_curvedist} and Lemma \ref{lemma_gamma} (iv) and (iii), we get
$$\cD\left(w_{j,k},\frac{1}{27}\right)\subset \cH\left(\lambda,\frac{\alpha}{|c_j|},\nu_5\right)$$
if $|k|$ is sufficiently large. By Koebe's $1/4$-theorem and Lemma \ref{lemma_lower_bound_derivative},
$$\psi_j^n\left(\cD\left(w_{j,k},\frac{1}{27}\right)\right)\supset \cD\left(\psi_j^n(w_{j,k}),\frac{|(\psi_j^n)'(w_{j,k})|}{4\cdot27}\right)\supset \cD\left(\psi_j^n(w_{j,k}),\frac{B}{4\cdot27}\right).$$
With $\teta:=B/(4\cdot27)$ we thus have
$$h_j^n(\cD(\psi_j^n(w_{j,k}),\teta))\subset \cD\left(w_{j,k},\frac{1}{27}\right)\subset q(\cF_j).$$
Since by Lemma \ref{lemma_invariance_Sj}, 
$$h_j^n(w)=(q\circ f\circ\vi_j)^n(w)=(q\circ f^n\circ\vi_j)(w)$$
for $w\in \cH(\lambda,\alpha/|c_j|,\nu_0)$, this implies
$$\cD(\psi_j^n(w_{j,k}),\teta)\subset q(\cF_j),$$
provided $|k|$ is sufficiently large.
\end{proof}

The final result of this section says that $q(\cF_j)$ has positive density in rectangles of sufficiently large side lengths that are  contained in $\cH(\lambda,1/|c_j|,\nu)$.

\begin{lemma}   \label{lemma_F_rechts}
There are $D_0,\nu,\eta_0>0$ such that for all $j\in\{1,...,d\}$ and any rectangle $\cR\subset \cH(\lambda,1/|c_j|,\nu)$ with sides parallel to the real and imaginary axis whose vertical and horizontal side lengths are both at least $D_0$, we have
$$\dens(q(\cF_j), \cR)\ge\eta_0.$$
\end{lemma}

\begin{proof}
First suppose that 
\begin{equation}  \label{eq_rectangle}
\cR=\{w:\,x_1\le\re w\le x_2,\,y_1\le\im \le y_2\}
\end{equation}
where
\begin{equation}  \label{eq_small_sidelength}
2\pi+2(C+\teta)\le x_2-x_1, y_2-y_1\le2(2\pi+2(C+\teta)),
\end{equation}
with $C$ as in Lemma \ref{lemma_preimages_seq} and $\teta$ as in Lemma \ref{lemma_disk_q(F)}.
Let $v_{j,k}$ be as in Lemma \ref{lemma_zeros} and $w_{j,k}=v_{j,k}+1/26$. There is $k\in\dZ$ such that $y_1+C+\teta\le\im w_{j,k}=\im v_{j,k}\le y_2-C-\teta$. Also, by Lemma \ref{lemma_preimages}, there is $n\in\dN$ such that $x_1+\teta<\re\psi_j^n(w_{j,k})<x_2-\teta$.  By Lemma \ref{lemma_preimages_seq}, we have $y_1+\teta\le\im\psi_j^n(w_{j,k})\le y_2-\teta$. Thus,
$$\cD(\psi_j^n(w_{j,k}),\teta)\subset \cR.$$ 
Also, by Lemma \ref{lemma_disk_q(F)}, 
$$\cD(\psi_j^n(w_{j,k}),\teta)\subset q(\cF_j).$$
We obtain
$$\dens(q(\cF_j), \cR)\ge\frac{\meas(\cD(\psi_j^n(w_{j,k}),\teta))}{\meas \cR}\ge\frac{\pi\teta^2}{4(2\pi+2(C+\teta))^2}=:\eta_0.$$
Now, if $\cR\subset \cH(\lambda,1/|c_j|,\nu)$ is any rectangle whose horizontal and vertical side length both exceed $D_0:=2\pi+2(C+\teta)$, then $\cR$ can be written as the union of rectangles of the form \eqref{eq_rectangle} that satisfy \eqref{eq_small_sidelength} and have pairwise disjoint interior, whence the claim follows.
\end{proof}


\section{The set \texorpdfstring{$q(\cF(f))$}{q(F(f))}: second part}   \label{sec_q(F)2}

In this section, we investigate the density of $q(\cF(f))$ in subsets of $\cH(\lambda-1,\alpha_1/|c_j|,\nu)\setminus \cH(\lambda,\beta_1/|c_j|,\nu)$ for small $\alpha_1>0$ and large $\beta_1>0$. 

We first give an approximate expression for  $h_j$ in $\cH(\lambda-1,\alpha_1/|c_j|,\nu)\setminus \cH(\lambda,\beta_1/|c_j|,\nu)$.

\begin{lemma}   \label{lemma_asymptotics_mitte}
Let $\epsi>0$ and $j\in\{1,...,d\}$. Then there are $\alpha_1,\beta_1,\nu>0$ such that for all $w\in \cH(\lambda-1,\alpha_1/|c_j|,\nu)\setminus \cH(\lambda,\beta_1/|c_j|,\nu)$, we have
$$\left|\frac{h_j(w)-w}{-c_je^{-w}\vi_j(w)^{d\lambda}}-1\right|<\epsi.$$
\end{lemma}

\begin{proof}
Taylor expansion of $q$ around $\vi_j(w)$ yields
\begin{equation}
h_j(w)=q(f(\vi_j(w)))=w+\sum_{k=1}^d\frac{q^{(k)}(\vi_j(w))}{k!}(f(\vi_j(w))-\vi_j(w))^k.
\end{equation}
Thus,
\begin{equation}  \label{eq_asymptmitte_approach}
\begin{split}
&\frac{h_j(w)-w}{-c_je^{-w}\vi_j(w)^{d\lambda}}-1\\
&=\frac{q'(\vi_j(w))(f(\vi_j(w))-\vi_j(w))}{-c_je^{-w}\vi_j(w)^{d\lambda}}-1+\sum_{k=2}^d\frac{q^{(k)}(\vi_j(w))}{k!}\frac{(f(\vi_j(w))-\vi_j(w))^k}{-c_je^{-w}\vi_j(w)^{d\lambda}}.
\end{split}
\end{equation}
By Corollary \ref{cor_asymptotics_f}, 
\begin{equation}  \label{eq_asympt_f_ref}
f(\vi_j(w))=\vi_j(w)-\frac{1}{q'(\vi_j(w))}\left(1+O\left(\frac{1}{|w|}\right)+c_je^{-w}\vi_j(w)^{d\lambda}\left(1+O\left(\frac{1}{|w|^{1/d}}\right)\right)\right)
\end{equation}
as $w\to\infty$. 
Hence,
\begin{equation}   \label{eq_asympt_summand2}
\frac{q'(\vi_j(w))(f(\vi_j(w))-\vi_j(w))}{-c_je^{-w}\vi_j(w)^{d\lambda}}-1=\frac{1+O(1/|w|)}{c_je^{-w}\vi_j(w)^{d\lambda}}+O\left(\frac{1}{|w|^{1/d}}\right).
\end{equation}
For $w\in \dC\setminus\cH(\lambda,\beta_1/|c_j|,\nu)$ with $|\im w|\ge\nu$, we have
$$\left|c_je^{-w}\vi_j(w)^{d\lambda}\right|\ge\frac{\beta_1}{2},$$
provided $\nu$ is sufficiently large. Inserting this into \eqref{eq_asympt_summand2} yields
\begin{equation}  \label{eq_summand2_final}
\left|\frac{q'(\vi_j(w))(f(\vi_j(w))-\vi_j(w))}{-c_je^{-w}\vi_j(w)^{d\lambda}}-1\right|\le\frac{3}{\beta_1}+O\left(\frac{1}{|w|^{1/d}}\right)<\frac{\epsi}{d}
\end{equation}
if $\beta_1$ and $|w|$ are sufficiently large.

Also, for $w\in\cH(\lambda-1,\alpha_1/|c_j|,\nu)$, we have
$$\left|c_je^{-w}\vi_j(w)^{d(\lambda-1)}\right|\le2\alpha_1,$$
provided $\nu$ is sufficiently large. By \eqref{eq_asympt_f_ref}, this yields
\begin{equation}  \label{eq_asymptmitte_est2}
|f(\vi_j(w))-\vi_j(w)|\le\frac{1}{|q'(\vi_j(w))|}\left(1+O\left(\frac{1}{|w|}\right)+3\alpha_1\vi_j(w)^d\right)\le\frac{4}{d}\alpha_1|\vi_j(w)|,
\end{equation}
provided $|w|$ is sufficiently large.
If $k\ge2$, then by \eqref{eq_summand2_final} and \eqref{eq_asymptmitte_est2}, we have
\begin{align}
&\left|\frac{q^{(k)}(\vi_j(w))}{k!}\cdot\frac{(f(\vi_j(w))-\vi_j(w))^k}{-c_je^{-w}\vi_j(w)^{d\lambda}}\right|\\
&=\left|\frac{q'(\vi_j(w))(f(\vi_j(w))-\vi_j(w))}{-c_je^{-w}\vi_j(w)^{d\lambda}}\right|\cdot\left|\frac{q^{(k)}(\vi_j(w))}{k!q'(\vi_j(w))}\right|\cdot|f(\vi_j(w))-\vi_j(w)|^{k-1}\\
&\le\left(1+\frac{\epsi}{d}\right)\cdot\left|\frac{q^{(k)}(\vi_j(w))}{k!q'(\vi_j(w))}\right|\cdot\left(\frac{4}{d}\alpha_1|\vi_j(w)|\right)^{k-1}\\
&\le\left(1+\frac{\epsi}{d}\right)\binom{d}{k}\frac{2}{d}|\vi_j(w)|^{-k+1}\left(\frac{4}{d}\alpha_1|\vi_j(w)|\right)^{k-1}\\
&=\left(1+\frac{\epsi}{d}\right)\binom{d}{k}\frac{2\cdot4^{k-1}}{d^k}\alpha_1^{k-1}<\frac{\epsi}{d} \label{eq_asymptmitte_final}
\end{align}
if $|w|$ is sufficiently large and $\alpha_1$ is sufficiently small. Inserting \eqref{eq_summand2_final} and \eqref{eq_asymptmitte_final} into \eqref{eq_asymptmitte_approach} yields the desired conclusion.
\end{proof}

We will now proceed as follows. First, we show that $h_j$ maps the intersection of certain horizontal strips with $\cH(\lambda-1,\alpha_1/|c_j|,\nu)\setminus\cH(\lambda,\beta_1/|c_j|,\nu)$ into $\cH(\lambda,1/c^*,\nu)$ where $c^*=\max_l|c_l|$. The idea is that if $\im w$ lies in certain intervals, then the argument of $-c_je^{-w}\vi_j(w)^{d\lambda}$ is small, and using that $h_j(w)\approx w-c_je^{-w}\vi_j(w)^{d\lambda}$ by Lemma \ref{lemma_asymptotics_mitte}, one can deduce that $\re h_j(w)$ is large. By Section \ref{sec_q(F)1}, the set $q(\cF(f))$ has positive density in large bounded subsets of $\cH(\lambda,1/c^*,\nu)$. Together with the invariance of $\cF(f)$ under $f$, we deduce that $q(\cF(f))$ has positive density in large bounded subsets of $\cH(\lambda-1,\alpha_1/|c_j|,\nu)\setminus\cH(\lambda,\beta_1/|c_j|,\nu)$. 

The next Lemma deals with the mapping behaviour of $f$ in certain horizontal strips in $\cH(\lambda-1,\alpha_1/|c_j|,\nu)\setminus\cH(\lambda,\beta_1/|c_j|,\nu)$.
 For $j\in\{1,...,d\}$ and $n\in\dZ$, let
$$y_{n}^j:=\begin{cases}\arg(-c_j)+\lambda(\pi/2+2\pi(j-1))+2n\pi&\text{if }n\ge0\\
\arg(-c_j)+\lambda(-\pi/2+2\pi j)+2n\pi&\text{if }n<0.\end{cases}$$

\begin{lemma}  \label{lemma_map_mitte}
Let $\epsi\in(0,\pi/4)$. Then there are $\alpha_1, \beta_1, \nu>0$ such that the following holds. Let $j\in\{1,...,d\}$. Suppose that $w$ lies in the closure of $\cH(\lambda-1,\alpha_1/|c_j|,\nu)\setminus\cH(\lambda,\beta_1/|c_j|,\nu)$ and there exists $n\in\dZ$ with 
$|\im w-y_n^j|\le\pi/4.$
Let $\beta\ge\beta_1$ such that $w\in\Gamma(\lambda,\beta/|c_j|)$, and let $\theta:=\im w-y_n^j$. 
Then,
$$|h_j(w)-w|\le(1+\epsi)\beta,$$
$$(1-\epsi)\beta\cos(|\theta|+\epsi)\le\re(h_j(w)-w)\le(1+\epsi)\beta$$
and
$$(1-\epsi)\beta\sin(|\theta|-\epsi)\le|\im(h_j(w)-w)|\le(1+\epsi)\beta\sin(|\theta|+\epsi).$$
\end{lemma}

\begin{proof}
 By Lemma \ref{lemma_asymptotics_mitte}, 
\begin{equation}  \label{eq_hilfslemma_asympt_mitte}
\left|\frac{h_j(w)-w}{-c_je^{-w}\vi_j(w)^{d\lambda}}-1\right|\le\frac{\epsi}{2},
\end{equation}
provided $\alpha_1$ is sufficiently small and $\beta_1$ and $\nu$ are sufficiently large. Thus,
$$\left(1-\frac{\epsi}{2}\right)\left|c_je^{-w}\vi_j(w)^{d\lambda}\right|\le|h_j(w)-w|\le\left(1+\frac{\epsi}{2}\right)\left|c_je^{-w}\vi_j(w)^{d\lambda}\right|.$$
Since $w\in\Gamma(\lambda,\beta/|c_j|)$, this yields
$$(1-\epsi)\beta\le|h_j(w)-w|\le(1+\epsi)\beta$$
if $\nu$ is sufficiently large. Also, by \eqref{eq_hilfslemma_asympt_mitte},
\begin{equation}  \label{eq_argasympt}
\left|\arg\left(\frac{h_j(w)-w}{-c_je^{-w}\vi_j(w)^{d\lambda}}\right)\right|\le\arcsin\left(\frac{\epsi}{2}\right)\le\frac{\pi}{4}\epsi.
\end{equation}
We have
$$\arg w=\arg q(\vi_j(w))\equiv\arg(\vi_j(w)^d(1+o(1)))\equiv d\arg\vi_j(w)+o(1)\mod2\pi$$
as $w\to\infty$. Since $w$ lies in the closure of $\cH(\lambda-1,\alpha_1/|c_j|,\nu)\setminus\cH(\lambda,\beta_1/|c_j|,\nu)$, we have
$$\arg w=\sgn(\im(w))\frac{\pi}{2}+o(1)=\sgn(n)\frac{\pi}{2}+o(1)$$
if $|n|$ is sufficiently large. Hence,
$$\arg\vi_j(w)\equiv\sgn(n)\frac{\pi}{2d}+o(1)\mod\frac{2\pi}{d}.$$
Since $\vi_j(w)\in \cS_j$, we obtain
$$\arg\vi_j(w)\equiv\begin{cases}\pi/(2d)+2\pi(j-1)/d+o(1)&\text{if }n>0\\
-\pi/(2d)+2\pi j/d+o(1)&\text{if }n<0\end{cases}\mod2\pi.$$
Thus,
\begin{equation}
\begin{split}
\arg\left(-c_je^{-w}\vi_j(w)^{d\lambda}\right)&\equiv\arg(-c_j)-\im(w)+d\lambda\arg\vi_j(w)\\
&\equiv-\theta-2n\pi+o(1)\equiv-\theta+o(1)\mod2\pi.
\end{split}
\end{equation}
By \eqref{eq_argasympt}, this implies
$$|\theta|-\epsi\le|\arg(h_j(w)-w)|\le|\theta|+\epsi$$
if $|w|$ is sufficiently large. We obtain
\begin{align*}
\re(h_j(w)-w)&\le|h_j(w)-w|\le(1+\epsi)\beta,\\
\re(h_j(w)-w)&=|h_j(w)-w|\cos(\arg(h_j(w)-w))\ge(1-\epsi)\beta\cos(|\theta|+\epsi),\\
|\im(h_j(w)-w)|&=|h_j(w)-w|\cdot|\sin(\arg(h_j(w)-w))|\le(1+\epsi)\beta\sin(|\theta|+\epsi),\\
|\im(h_j(w)-w)|&=|h_j(w)-w|\cdot|\sin(\arg(h_j(w)-w))|\ge(1-\epsi)\beta\sin(|\theta|-\epsi).\qedhere
\end{align*}
\end{proof}

Let
\begin{equation}  \label{eq_cstar}
c^*:=\max_{1\le l\le d}|c_l|.
\end{equation}
The following Lemma says that $h_j$ maps the intersection of $\cH(\lambda-1,\alpha_1/|c_j|,\nu)\setminus\cH(\lambda,\beta_1/|c_j|,\nu)$ with certain horizontal strips into $\cH(\lambda,1/c^*,\nu)$.

\begin{lemma}   \label{lemma_h(Q)inH}
There are $\alpha_1,\beta_1,\nu>0$ such that for all $j\in\{1,...,d\}, n\in\dZ$ and all $w$ in the closure of $\cH(\lambda-1,\alpha_1/|c_j|,\nu)\setminus\cH(\lambda,\beta_1/|c_j|,\nu)$ with $|\im w-y_n^j|\le\pi/4$, we have $h_j(w)\in\cH(\lambda,1/c^*,\nu)$.
\end{lemma}

\begin{proof}
We have $w\in\Gamma(\lambda,\beta/|c_j|)$ for some $\beta\ge\beta_1$. Since also $w\in\cH(\lambda-1,\alpha_1/|c_j|,\nu)$, we have
$$\lambda\log|w|-\log\beta=\re w\ge(\lambda-1)\log|w|-\log\alpha_1$$
and hence
\begin{equation}  \label{eq_beta}
\beta\le\alpha_1|w|.
\end{equation}
Let $\epsi\in(0,\pi/4)$, $\theta:=\im w-y_n^j$, and suppose that  $\alpha_1,\beta_1,\nu$ are chosen such that the conclusion of Lemma \ref{lemma_map_mitte} holds. Then
$$|\im h_j(w)-\im w|\le(1+\epsi)\beta\sin(|\theta|+\epsi)\le2\beta.$$
By \eqref{eq_beta} and since $|\im w|=(1+o(1))|w|$ for $w$ in the closure of $\cH(\lambda-1,\alpha_1/|c_j|,\nu)\setminus\cH(\lambda,\beta_1/|c_j|,\nu)$, this yields
\begin{equation}
|\im h_j(w)|\ge|\im w|-2\beta=(1+o(1))|w|-2\beta\ge(1+o(1))|w|-2\alpha_1|w|>\nu,
\end{equation}
provided $|w|$ is sufficiently large and $\alpha_1$ is sufficiently small.

Also, by Lemma \ref{lemma_map_mitte} and \eqref{eq_beta},
\begin{equation}  
|h_j(w)-w|\le(1+\epsi)\beta\le(1+\epsi)\alpha_1|w|.
\end{equation}
If $\alpha_1$ is sufficiently small, we obtain
$$\frac{1}{2}|w|\le|h_j(w)|\le2|w|$$
and hence
\begin{equation}   \label{eq_modh-w}
\frac{1}{2}|h_j(w)|\le|w|\le2|h_j(w)|.
\end{equation}
By Lemma \ref{lemma_map_mitte} and \eqref{eq_modh-w},
\begin{equation}
\begin{split}
&\re h_j(w)\ge\re w+(1-\epsi)\beta\cos\left(\frac{\pi}{4}+\epsi\right)\\
&=\lambda\log|w|-\log\beta+(1-\epsi)\beta\cos\left(\frac{\pi}{4}+\epsi\right)\\
&\ge\lambda\log|h_j(w)|-|\lambda|\log2-\log\beta+(1-\epsi)\beta\cos\left(\frac{\pi}{4}+\epsi\right)\\
&\ge\lambda\log|h_j(w)|-\log\frac{1}{c^*}
\end{split}
\end{equation}
if $\beta_1$ and hence $\beta$ is sufficiently large. Thus, $h_j(w)\in\cH(\lambda,1/c^*, \nu).$
\end{proof}

Let us now define several sets. We start with subsets $\cQ_{n,k}^j,\tilde\cQ_{n,k}^j\subset\dC\setminus\cH(\lambda,\beta_1/|c_j|,\nu)$ for $j\in\{1,...,d\}$, $k\in\dN$ and $n\in\dZ$. 

Let $0<\theta_1<1/(6\pi)\arccos(5/6)$. For $j\in\{1,...,d\}$, $k\in\dN$ and $n\in\dZ$, let $\cQ_{n,k}^j$ be the set of all 
$$w\in \cH\left(\lambda, \frac{2^{k+1}\beta_1}{|c_j|},\nu\right)\setminus \cH\left(\lambda,\frac{2^k\beta_1}{|c_j|},\nu\right)$$
such that
$$|\im w-y_{n}^j|\le\theta_1.$$
Also, let $\tilde \cQ_{n,k}^j$ be the set of all 
$$w\in \cH\left(\lambda,\frac{2^{k+2}\beta_1}{|c_j|},\nu\right)\setminus \cH\left(\lambda,\frac{2^{k-1}\beta_1}{|c_j|},\nu\right)$$
such that 
$$|\im w-y_{n}^j|\le5\pi\theta_1.$$

See Figure \ref{fig_Q} for an illustration of these sets. Note that $\cQ_{n,k}^j\subset\tilde\cQ_{n,k}^j$. If $\tilde \cQ_{n,k}^j\subset\cH(\lambda-1,\alpha_1/|c_j|,\nu)$, then by Lemma \ref{lemma_h(Q)inH}, we have $h_j(\tilde Q_{n,k}^j)\subset\cH(\lambda,1/c^*,\nu)$.

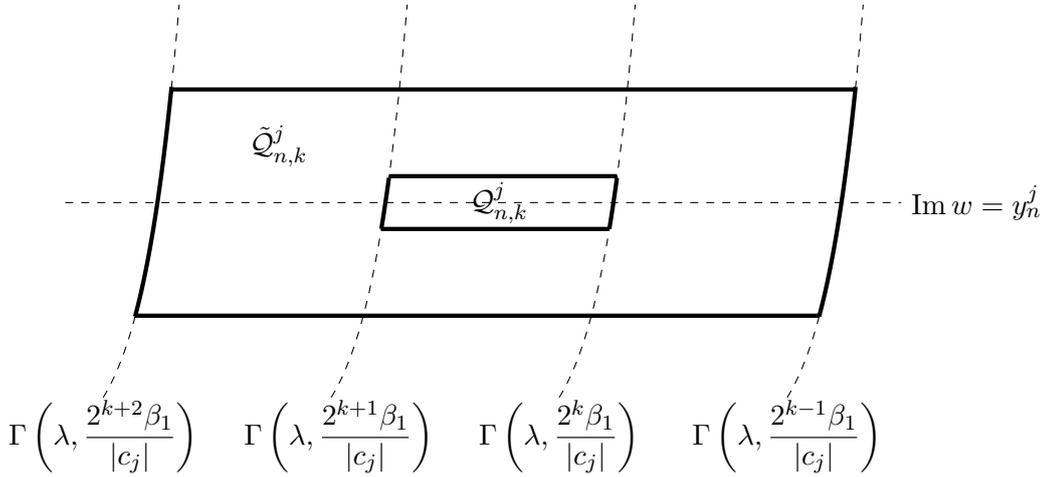
\begin{figure}[ht]
\centering
\begin{tikzpicture}
\draw[dashed, smooth, samples = 20, domain = 0.5:1.5] plot(\x, {0.3*exp(2*\x)});
\draw[dashed, smooth, samples = 20, domain = 3.5:4.5] plot(\x, {0.3*exp(2*(\x-3))});
\draw[dashed, smooth, samples = 20, domain = 6.5:7.5] plot(\x, {0.3*exp(2*(\x-6))});
\draw[dashed, smooth, samples = 20, domain = 9.5:10.5] plot(\x, {0.3*exp(2*(\x-9))});

\draw (0.5,0.92) node[below]{$\Gamma\left(\lambda,\dfrac{2^{k+2}\beta_1}{|c_j|}\right)$};
\draw (3.6, 0.92) node[below]{$\Gamma\left(\lambda,\dfrac{2^{k+1}\beta_1}{|c_j|}\right)$};
\draw (6.5, 0.92) node[below]{$\Gamma\left(\lambda,\dfrac{2^{k}\beta_1}{|c_j|}\right)$};
\draw (9.5, 0.92) node[below]{$\Gamma\left(\lambda, \dfrac{2^{k-1}\beta_1}{|c_j|}\right)$};

\draw[dashed](0, 3.4) -- (11,3.4);
\draw (11, 3.4) node[right]{$\im w=y_n^j$};

\draw[ultra thick, smooth, samples = 10, domain = 4.16:4.26] plot(\x, {0.3*exp(2*(\x-3))});
\draw[ultra thick, smooth, samples = 10, domain = 7.16:7.26] plot(\x, {0.3*exp(2*(\x-6))});
\draw[ultra thick] (4.16, 3.05) -- (7.16, 3.05);
\draw[ultra thick] (4.26, 3.75) -- (7.26, 3.75);

\draw[ultra thick, smooth, samples = 10, domain = 0.92:1.4] plot(\x, {0.3*exp(2*\x)});
\draw[ultra thick, smooth, samples = 10, domain = 9.92:10.4] plot(\x, {0.3*exp(2*(\x-9))});
\draw[ultra thick] (0.92, 1.9) -- (9.92, 1.9);
\draw[ultra thick] (1.4, 4.9) -- (10.4, 4.9);

\draw (2.81, 4.15) node{$\tilde \cQ_{n,k}^j$};
\draw (5.71, 3.4) node{$\cQ_{n,k}^j$};

\end{tikzpicture}
\caption{An illustration of the sets $Q_{n,k}^j$ and $\tilde Q_{n,k}^j$.}
\label{fig_Q}
\end{figure}

 Moreover, let $\cR_{n,k}^j$ be the rectangle containing all $v\in\dC$ satisfying
$$\frac{3}{4}2^k\beta_1<\re v-\lambda\log|n|<\frac{5}{2}2^{k}\beta_1$$
and 
$$|\im v-y_{n}^j|<3\cdot2^{k}\beta_1\theta_1.$$
Also, let $\tilde \cR_{n,k}^j$ be the rectangle containing all $v\in\dC$ satisfying
$$\frac{5}{8}2^{k}\beta_1<\re v-\lambda\log|n|<3\cdot2^{k}\beta_1$$
and
$$|\im v-y_{n}^j|<4\cdot2^{k}\beta_1\theta_1.$$
Note that $\cR_{n,k}^j\subset\tilde\cR_{n,k}^j$.

\begin{lemma}   \label{lemma_mapQ_mitte}
There are $\alpha_1, \beta_1, \nu, n_0>0$ such that the following holds. If $j\in\{1,...,d\}$, $n\in\dZ$ with $|n|\ge n_0$ and $k\in\dN$ are such that 
$\tilde \cQ_{n,k}^j\subset \cH(\lambda-1,\alpha_1/|c_j|,\nu),$
then 
$h_j(\cQ_{n,k}^j)\subset \cR_{n,k}^j$
and
$h_j(\tilde \cQ_{n,k}^j)\supset\tilde \cR_{n,k}^j.$
\end{lemma}

\begin{proof}
For $w\in\cH(\lambda-1,\alpha_1/|c_j|,\nu)\setminus\cH(\lambda,\beta_1/|c_j|,\nu)$, we have $\re w=o(|w|)$ and hence $|w|=(1+o(1))|\im w|$ as $w\to\infty$. If $|\im w-y_n^j|\le5\pi\theta_1$ and $|n|$ is sufficiently large, we obtain
\begin{equation}  \label{eq_modw}
|n|\le|w|\le e^2|n|.
\end{equation}
For $w\in\tilde \cQ_{n,k}^j$, this implies that
\begin{equation}   \label{eq_lowerrealestimate}
\re w\ge\lambda\log|w|-\log\frac{2^{k+2}\beta_1}{|c_j|}\ge\lambda\log|n|-2|\lambda|-\log\frac{2^{k+2}\beta_1}{|c_j|}
\end{equation}
and
\begin{equation}  \label{eq_upperrealestimate}
\re w\le\lambda\log|w|-\log\frac{2^{k-1}\beta_1}{|c_j|}\le\lambda\log|n|+2|\lambda|-\log\frac{2^{k-1}\beta_1}{|c_j|}.
\end{equation}
Let $\epsi>0$ be small and let $w\in \cQ_{n,k}^j\subset\tilde \cQ_{n,k}^j$. By Lemma \ref{lemma_map_mitte} and \eqref{eq_lowerrealestimate},
and since $0<\theta_1<1/(6\pi)\arccos(5/6)<(1/2)\arccos(5/6)$, we have
\begin{equation}
\begin{split}
\re h_j(w)&\ge\re w+(1-\epsi)2^k\beta_1\cos(\theta_1+\epsi)\\
&\ge\re w+(1-\epsi)2^k\beta_1\cos(2\theta_1)\\
&>\lambda\log|n|-2|\lambda|-\log\frac{2^{k+2}\beta_1}{|c_j|}+(1-\epsi)2^k\beta_1\cdot\frac{5}{6}\\
&>\lambda\log|n|+\frac{3}{4}2^k\beta_1
\end{split}
\end{equation}
if $\epsi$ is sufficiently small and $\beta_1$ is sufficiently large. Analogously,
\begin{equation}
\begin{split}
\re h_j(w)&\le\re w+(1+\epsi) 2^{k+1}\beta_1\\
&\le\lambda\log|n|+2|\lambda|-\log\frac{2^{k-1}\beta_1}{|c_j|}+(1+\epsi)2^{k+1}\beta_1\\
&<\lambda\log|n|+\frac{5}{2}2^{k}\beta_1
\end{split}
\end{equation}
if $\epsi$ is sufficiently small and $\beta_1$ is sufficiently large. Moreover, by Lemma \ref{lemma_map_mitte},
\begin{equation}
\begin{split}
|\im h_j(w)-y_{n}^j|&\le|\im h_j(w)-\im w|+|\im w-y_n^j|\\
&\le(1+\epsi)2^{k+1}\beta_1\sin(\theta_1+\epsi)+|\im w-y_{n}^j|\\
&\le(1+\epsi)2^{k+1}\beta_1(\theta_1+\epsi)+\theta_1\\
&<3\cdot2^{k}\beta_1\theta_1
\end{split}
\end{equation}
if $\epsi$ is sufficiently small and $\beta_1$ is sufficiently large. Thus, $h_j(\cQ_{n,k}^j)\subset \cR_{n,k}^j$.

In the following, we show that $h_j(\partial\tilde \cQ_{n,k}^j)\cap \tilde \cR_{n,k}^j=\emptyset$. Since we have already shown that $h_j(\tilde \cQ_{n,k}^j)\cap\tilde \cR_{n,k}^j\ne\emptyset$, this implies that $h_j(\tilde \cQ_{n,k}^j)\supset \tilde \cR_{n,k}^j.$

If $w\in\Gamma(\lambda,2^{k-1}\beta_1/|c_j|)$ and $\beta_1$ is large, then by Lemma \ref{lemma_map_mitte} and \eqref{eq_upperrealestimate},
\begin{equation}
\begin{split}
\re h_j(w)&\le\re w+(1+\epsi)2^{k-1}\beta_1\\
&\le\lambda\log|n|+2|\lambda|-\log\frac{2^{k-1}\beta_1}{|c_j|}+(1+\epsi)2^{k-1}\beta_1\\
&<\lambda\log|n|+\frac{5}{8}2^{k}\beta_1.
\end{split}
\end{equation}
If $w\in\Gamma(\lambda,2^{k+2}\beta_1/|c_j|)$ and $|\im w-y_{n}^j|\le5\pi\theta_1$, then by Lemma \ref{lemma_map_mitte} and \eqref{eq_lowerrealestimate}, and since $0<\theta_1<1/(6\pi)\arccos(5/6)$, we have
\begin{equation}
\begin{split}
\re h_j(w)&\ge \re w+(1-\epsi)2^{k+2}\beta_1\cos(5\pi\theta_1+\epsi)\\
&\ge\lambda\log|n|-2|\lambda|-\log\frac{2^{k+2}\beta_1}{|c_j|}+(1-\epsi)2^{k+2}\beta_1\cos(6\pi\theta_1)\\
&>\lambda\log|n|-2|\lambda|-\log\frac{2^{k+2}\beta_1}{|c_j|}+(1-\epsi)2^{k+2}\beta_1\cdot\frac{5}{6}\\
&>\lambda\log|n|+3\cdot2^{k}\beta_1,
\end{split}
\end{equation}
provided $\epsi$ is sufficiently small and $\beta_1$ is sufficiently large.

If $|\im w-y_{n}^j|=5\pi\theta_1$ and $w\in \cH(\lambda,2^{k+2}\beta_1/|c_j|,\nu)\setminus \cH(\lambda,2^{k-1}\beta_1/|c_j|,\nu)$, then by Lemma \ref{lemma_map_mitte},
\begin{equation}
\begin{split}
|\im h_j(w)-y_{n}^j|&\ge|\im h_j(w)-\im(w)|-|\im w-y_{n}^j|\\
&\ge(1-\epsi)2^{k-1}\beta_1\sin(5\pi\theta_1-\epsi)-5\pi\theta_1\\
&\ge(1-\epsi)2^{k-1}\beta_1\frac{2}{\pi}(5\pi\theta_1-\epsi)-5\pi\theta_1\\
&>4\cdot2^{k}\beta_1\theta_1,
\end{split}
\end{equation}
provivded $\epsi$ is sufficiently small and $\beta_1$ is sufficiently large.
Thus, $h_j(\partial\tilde \cQ_{n,k}^j)\subset\dC\setminus\tilde \cR_{n,k}^j$. 
\end{proof}

Next, we prove that the density of $q(\cF_j)$ in $\tilde Q_{n,k}^j$ is bounded below by a positive constant.

\begin{lemma}   \label{lemma_FinQ_mitte}
There are $\alpha_1,\beta_1,\nu, \delta, n_1>0$ such that for all $j\in\{1,...,d\}$, $n\in\dZ$ with $|n|\ge n_1$ and $k\in\dN$ with $\tilde\cQ_{n,k}^j\subset \cH(\lambda-1,\alpha_1/|c_j|,\nu)$, we have
$$\dens(q(\cF_j), \tilde \cQ_{n,k}^j)\ge\delta.$$
\end{lemma}

\begin{proof}
By Section \ref{sec_singularities}, in particular Lemma \ref{lemma_zeros}, the function $h_j=q\circ f\circ\vi_j$ has no critical points in $\tilde \cQ_{n,k}^j$ if $\nu$ and $\beta_1$ are sufficiently large. By Lemma \ref{lemma_mapQ_mitte}, $h_j(\tilde \cQ_{n,k}^j)\supset\tilde \cR_{n,k}^j$ and $h_j(\cQ_{n,k}^j)\subset \cR_{n,k}^j$. Let $\cU$ be the component of $h_j^{-1}(\tilde \cR_{n,k}^j)$ containing $\cQ_{n,k}^j$. Then $\cQ_{n,k}^j\subset \cU\subset\tilde \cQ_{n,k}^j$. Since $\tilde \cR_{n,k}^j$ is simply connected, $h_j$ maps $\cU$ conformally onto $\tilde \cR_{n,k}^j$. Let $\psi:\,\tilde \cR_{n,k}^j\to \cU$ be the corresponding inverse function. By Lemma \ref{lemma_h(Q)inH},
$$h_j(\tilde \cQ_{n,k}^j)\subset \cH\left(\lambda,\frac{1}{c^*},\nu\right)\subset\dC\setminus (\overline{\cD(0,R)}\cup [0,\infty))=q(\cS_l)$$
for all $l\in\{1,...,d\}$.
Hence, there exists $l\in\{1,...,d\}$ such that $(f\circ\vi_j)(\tilde \cQ_{n,k}^j)\subset \cS_l$. We have $\psi(q(\cF_l)\cap~\tilde \cR_{n,k}^j)=q(\cF_j)\cap \cU$.  By the Koebe distortion theorem,  $\psi$ has bounded distortion in $\cR_{n,k}^j$ independent of $n$, $k$ and $j$.  We obtain
\begin{equation}
\begin{split}
\dens(q(\cF_j), \tilde \cQ_{n,k}^j)&\ge\dens(q(\cF_j), \psi(\cR_{n,k}^j))\cdot\dens(\psi(\cR_{n,k}^j), \tilde \cQ_{n,k}^j)\\
&=\dens(\psi(q(\cF_l)\cap\cR_{n,k}^j), \psi(\cR_{n,k}^j))\cdot\dens(\psi(\cR_{n,k}^j), \tilde \cQ_{n,k}^j)\\
&\ge c\dens(q(\cF_l), \cR_{n,k}^j)\cdot\dens(\cQ_{n,k}^j, \tilde \cQ_{n,k}^j)
\end{split}
\end{equation}
for some $c>0$ independent of $n,k$ and $j$. If $\beta_1$ is sufficiently large, then by Lemma \ref{lemma_F_rechts}, 
\begin{equation}
\dens(q(\cF_l), \cR_{n,k}^j)\ge \eta_0.
\end{equation}
Moreover, by Lemma \ref{lemma_gamma}, 
$$\meas \cQ_{n,k}^j\ge\frac{2}{3}\log2\cdot2\theta_1\quad\text{ and } \meas\tilde \cQ_{n,k}^j\le2\log 8\cdot10\pi\theta_1.$$
Hence,
$$\dens(q(\cF_j), \tilde \cQ_{n,k}^j)\ge c\eta_0\frac{\log2}{15\pi\log8}=:\delta.\qedhere$$
\end{proof}

The last Lemma of this Section says that there is a positive lower bound for the density of $q(\cF_j)$ in any sufficiently large rectangle contained in $\cH(\lambda-1,\alpha_1/|c_j|,\nu)\setminus\cH(\lambda,\beta_1/|c_j|,\nu).$

\begin{lemma}  \label{lemma_F_mitte}
There are $\alpha_1, \beta_1, \nu, D_1,\eta_1>0$ such that for all $j\in\{1,...,d\}$ and any rectangle $\cR\subset \cH(\lambda-1,\alpha_1/|c_j|,\nu)\setminus \cH(\lambda,\beta_1/|c_j|,\nu)$ with sides parallel to the real and imaginary axis and side lengths at least $D_1$, we have
$$\dens(q(\cF_j), \cR)\ge\eta_1.$$
\end{lemma}

\begin{proof}
Suppose that
$$D_1\ge5\log 8,$$
$$D_1\ge y_{n_1}^l+2\pi+10\pi\theta_1 \text{ and }D_1\ge |y_{-n_1}^l| +2\pi+10\pi\theta_1\text{ for all } l\in\{1,...,d\},$$
 with $n_1$ as in Lemma \ref{lemma_FinQ_mitte}. Let $\cR\subset \cH(\lambda-1,\alpha_1/|c_j|,\nu)\setminus \cH(\lambda,\beta_1/|c_j|,\nu)$ be a rectangle with sides parallel to the real and imaginary axis and side lengths at least $D_1$. By the definition of $\tilde \cQ_{n,k}^j$ and Lemma \ref{lemma_gamma}, there are $k\in\dN$ and $n\in\dZ$ with $|n|\ge n_1$ such that
$$\tilde \cQ_{n,k}^j\subset \cR.$$
If, in addition, the side lengths of $\cR$ do not exceed $2D_1$, then by Lemma \ref{lemma_FinQ_mitte} and Lemma \ref{lemma_gamma},
$$\dens(q(\cF_j),\cR)\ge\dens(q(\cF_j), \tilde \cQ_{n,k}^j)\cdot\dens(\tilde \cQ_{n,k}^j,\cR)\ge\delta\frac{2/3\log8\cdot10\pi\theta_1}{4D_1^2}.$$
Since any general rectangle with side lengths at least $D_1$ can be written as the union of rectangles with side lengths between $D_1$ and $2D_1$ which are disjoint up to the boundary, the claim follows.
\end{proof}


\section{The set \texorpdfstring{$q(\cF(f))$}{q(F(f))}: third part}   \label{sec_q(F)3}

For $\nu>0$, let
$$\cG_\nu:=\{w:\,|\im w|\ge\nu\}.$$
In this section, we investigate the density of $q(\cF(f))$ in subsets of $\cG_\nu\setminus\cH(\lambda-1,\beta_2/|c_j|,\nu)$ for large $\beta_2>0$. First, we give an approximation for $h_j$ in $\cG_\nu\setminus\cH(\lambda-1,\beta_2/|c_j|,\nu)$. 

\begin{lemma}   \label{lemma_asymptotics_links}
Let $\epsi>0$ and $j\in\{1,...,d\}$. Then there are $\beta_2,\nu>0$ such that for all $w\in\cG_\nu\setminus \cH(\lambda-1,\beta_2/|c_j|,\nu)$, we have
$$\left|\frac{h_j(w)}{(-c_j/d)^de^{-dw}w^{-m}}-1\right|<\epsi.$$
\end{lemma}

\begin{proof}
By Corollary \ref{cor_asymptotics_f},
\begin{align}
f(\vi_j(w))&=\vi_j(w)-\frac{1}{q'(\vi_j(w))}\left(1+O\left(\frac{1}{|w|}\right)\right)-\frac{c_je^{-w}}{p(\vi_j(w))}\\
&=O(|w|^{1/d})-\frac{c_j}{d}e^{-w}\vi_j(w)^{-m}\left(1+O\left(\frac{1}{|w|^{1/d}}\right)\right)  \label{eq_f_links}
\end{align}
as $w\to\infty$. Note that the $O(\cdot)$-terms do not depend on $\beta_2$. For $w\in \cG_\nu\setminus \cH(\lambda-1,\beta_2/|c_j|,\nu)$, we have
\begin{equation}   \label{eq_preflarge_links}
\left|\frac{c_j}{d}e^{-w}\vi_j(w)^{-m}\right|=\left|\frac{c_j}{d}e^{-w}\right|\cdot|w|^{\lambda-1+1/d}(1+o(1))\ge\frac{\beta_2|w|^{1/d}}{2d}
\end{equation}
if $|w|$ is sufficiently large. In particular, 
$$|f(\vi_j(w))|\ge\frac{\beta_2}{4d}|w|^{1/d}$$
if $\beta_2$ and $|w|$ are sufficiently large, and hence
$$h_j(w)=q(f(\vi_j(w)))=f(\vi_j(w))^d\left(1+O\left(\frac{1}{|w|^{1/d}}\right)\right)$$
as $w\to\infty$ in $\cG_\nu\setminus \cH(\lambda-1,\beta_2/|c_j|,\nu)$. Also, by \eqref{eq_f_links} and \eqref{eq_preflarge_links},
\begin{equation}
\begin{split}
\left|\frac{f(\vi_j(w))}{(-c_j/d)e^{-w}\vi_j(w)^{-m}}-1\right|&=\left|\frac{O(|w|^{1/d})}{(-c_j/d)e^{-w}\vi_j(w)^{-m}}+O\left(\frac{1}{|w|^{1/d}}\right)\right|\\
&\le\frac{2d}{\beta_2}O(1)+O\left(\frac{1}{|w|^{1/d}}\right),
\end{split}
\end{equation}
where the $O(\cdot)$-terms do not depend on $\beta_2$. Hence, we can achieve that 
$$\left|\frac{f(\vi_j(w))^d}{((-c_j/d)e^{-w}\vi_j(w)^{-m})^d}-1\right|\le\frac{\epsi}{2}$$
by taking $\beta_2$ and $\nu$ sufficiently large. Also,
$$\left(-\frac{c_j}{d}e^{-w}\vi_j(w)^{-m}\right)^d=\left(-\frac{c_j}{d}\right)^de^{-dw}w^{-m}\left(1+O\left(\frac{1}{|w|^{1/d}}\right)\right)$$
as $w\to\infty$, whence the claim follows.
\end{proof}

We proceed similarly as in Section \ref{sec_q(F)2}, that is, we show that $h_j$ maps certain subsets of $\cG_\nu\setminus\cH(\lambda-1,\beta_2/|c_j|,\nu)$ into $\cH(\lambda,1/c^*,\nu)$. We then apply the results of Section \ref{sec_q(F)1} to show that $q(\cF(f))$ has positive density in large bounded subsets of $\cG_\nu\setminus\cH(\lambda-1,\beta_2/|c_j|,\nu).$

For $n\in\dZ,\,k\in\dN$ and $j\in\{1,...,d\}$, let $\cP_{n,k}^j$ be the set of all 
$$w\in \cH\left(\lambda-1,\frac{2^{k+2}\beta_2}{|c_j|},\nu\right)\setminus \cH\left(\lambda-1,\frac{2^{k-1}\beta_2}{|c_j|},\nu\right)$$
satisfying
$$\frac{(2n-1)\pi}{d}\le\im w\le\frac{2(n+1)\pi}{d}.$$
There are $\theta_{n,k}^j\in[-\pi,\pi)$ and $r_{n,k}^j>0$ such that for all $w\in \cP_{n,k}^j$, we have 
$$|w|=r_{n,k}^j(1+o(1))\quad\text{and }\arg(w)=\theta_{n,k}^j+o(1)$$
as $|n|\to\infty$. Let $t_{n,k}^j\in[2n\pi/d,2(n+1)\pi/d)$ with
$$t_{n,k}^j\equiv \arg(-c_j)-\frac{m}{d}\theta_{n,k}^j\mod \frac{2\pi}{d}.$$

\begin{lemma}   \label{lemma_map_links}
Let $\theta^*\in(0,\pi/(4d))$. Then there are $\beta_2,\nu>0$ such that the following holds. Let $j\in\{1,...,d\}$, $k\in\dN$ and $w\in\cH(\lambda-1,2^{k+2}\beta_2/|c_j|,\nu)\setminus\cH(\lambda-1,2^{k-1}\beta_2/|c_j|,\nu)$ such that there exists $n\in\dZ$ with $t_{n,k}^j-\pi/(4d)\le\im w\le t_{n,k}^j-\theta^*$. Let $\beta\in[2^{k-1}\beta_2,2^{k+2}\beta_2]$ such that $w\in\Gamma(\lambda-1,\beta/|c_j|)$ and let $\theta:=t_{n,k}^j-\im w$.
Then
\begin{equation}   \label{eq_re_links}
\frac{3}{4}\left(\frac{\beta}{d}\right)^dr_{n,k}^j\cos(2d\theta)<\re h_j(w)<\frac{5}{4}\left(\frac{\beta}{d}\right)^dr_{n,k}^j
\end{equation}
and
\begin{equation}   \label{eq_im_links}
\frac{3}{4\pi}\left(\frac{\beta}{d}\right)^dr_{n,k}^jd\theta<\im h_j(w)<\frac{5}{2}\left(\frac{\beta}{d}\right)^dr_{n,k}^jd\theta.
\end{equation}
\end{lemma}

\begin{proof}
Let $\epsi>0$ be small. By Lemma \ref{lemma_asymptotics_links},
$$\left|\frac{h_j(w)}{(-c_j/d)^de^{-dw}w^{-m}}-1\right|<\epsi$$
if $\beta_2$ and $\nu$ are sufficiently large. Thus,
\begin{equation}  \label{eq_mod_links}
\left(1-\epsi\right)\left(\frac{|c_j|}{d}\right)^de^{-d\re w}|w|^{-m}\le|h_j(w)|\le\left(1+\epsi\right)\left(\frac{|c_j|}{d}\right)^de^{-d\re w}|w|^{-m}.
\end{equation}
Since $w\in\Gamma(\lambda-1,\beta/|c_j|)$, we have
$$|w|^{-1-m}e^{-d\re w}=|w|^{d(\lambda-1)}e^{-d\re w}=\left(\frac{\beta}{|c_j|}\right)^d.$$
Thus,
$$\left(\frac{|c_j|}{d}\right)^de^{-d\re w}|w|^{-m}=\left(\frac{\beta}{d}\right)^d|w|=\left(\frac{\beta}{d}\right)^dr_{n,k}^j(1+o(1))$$ as $|n|\to\infty$. Inserting the last equation into \eqref{eq_mod_links} yields
\begin{equation}  \label{eq_mod_links_detailed}
\frac{3}{4}\left(\frac{\beta}{d}\right)^dr_{n,k}^j<|h_j(w)|<\frac{5}{4}\left(\frac{\beta}{d}\right)^dr_{n,k}^j
\end{equation}
if $\epsi$ is sufficiently small and $|n|$ is sufficiently large.
Also, by Lemma \ref{lemma_asymptotics_links},
\begin{equation}  \label{eq_arg_links}
\left|\arg (h_j(w))-\arg\left(\left(-\frac{c_j}{d}\right)^de^{-dw}w^{-m}\right)\right|<\arcsin(\epsi)\le\frac{\pi}{2}\epsi.
\end{equation}
We have
\begin{equation}
\begin{split}
\arg\left(\left(-\frac{c_j}{d}\right)^de^{-dw}w^{-m}\right)&\equiv d\arg(-c_j)-d\im w-m\arg w\\
&\equiv d\arg(-c_j)-dt_{n,k}^j+d\theta-m\theta_{n,k}^j+o(1)\\
&\equiv d\theta+o(1) \mod 2\pi
\end{split}
\end{equation}
as $|n|\to\infty$. By \eqref{eq_arg_links}, this yields
\begin{equation}  \label{eq_arg_links_detailed}
\frac{d\theta}{2}<\arg h_j(w)<2d\theta
\end{equation}
if $\epsi$ is sufficiently small compared to $\theta^*$. By \eqref{eq_mod_links_detailed}, \eqref{eq_arg_links_detailed} and the fact that $(2/\pi)x\le\sin x\le x$ for $0\le x\le\pi/2$, we obtain
\begin{equation}
\begin{split}
\re h_j(w)&\le|h_j(w)|<\frac{5}{4}\left(\frac{\beta}{d}\right)^dr_{n,k}^j,\\
\re h_j(w)&=|h_j(w)|\cos(\arg h_j(w))>\frac{3}{4}\left(\frac{\beta}{d}\right)^dr_{n,k}^j\cos(2d\theta),\\
\im h_j(w)&=|h_j(w)|\sin(\arg h_j(w))<\frac{5}{4}\left(\frac{\beta}{d}\right)^dr_{n,k}^j\sin(2d\theta)\le\frac{5}{2}\left(\frac{\beta}{d}\right)^dr_{n,k}^jd\theta,\\
\im h_j(w)&=|h_j(w)|\sin(\arg h_j(w))>\frac{3}{4}\left(\frac{\beta}{d}\right)^dr_{n,k}^j\sin\left(\frac{d\theta}{2}\right)\ge\frac{3}{4\pi}\left(\frac{\beta}{d}\right)^dr_{n,k}^jd\theta. \qedhere
\end{split}
\end{equation}
\end{proof}

Let us now define several sets. We start with subsets $\cT_{n,k}^j,\tilde T_{n,k}^j\subset\cG_\nu\setminus\cH(\lambda-1,\beta_2/|c_j|,\nu)$.  Let 
$$0<\theta_2<\frac{1}{2\cdot4^{d+1}d\pi}\arccos\left(\frac{11}{12}\right).$$
For $n\in\dZ$, $k\in\dN$ and $j\in\{1,...,d\}$, let $\cT_{n,k}^j$ be the set of all 
$$w\in \cH\left(\lambda-1,\frac{2^{k+1}\beta_2}{|c_j|},\nu\right)\setminus \cH\left(\lambda-1,\frac{2^k\beta_2}{|c_j|},\nu\right)$$
satisfing
$$t_{n,k}^j-\theta_2\le\im w\le t_{n,k}^j-\frac{\theta_2}{2}.$$
Also, let $\tilde \cT_{n,k}^j$ be the set of all 
$$w\in \cH\left(\lambda-1,\frac{2^{k+2}\beta_2}{|c_j|},\nu\right)\setminus \cH\left(\lambda-1,\frac{2^{k-1}\beta_2}{|c_j|},\nu\right)$$
satisfying
$$t_{n,k}^j-4^{d+1}\pi\theta_2\le\im w\le t_{n,k}^j-\frac{1}{10\cdot4^d\pi}\theta_2.$$
Note that $\cT_{n,k}^j\subset\tilde \cT_{n,k}^j$. See Figure \ref{fig_T} for an illustration of $\cT_{n,k}^j$ and $\tilde\cT_{n,k}^j$.

\begin{figure}[ht]
\centering
\begin{tikzpicture}
\draw[dashed, smooth, samples = 20, domain = -1.5:-0.5] plot(\x, {0.3*exp(-2*\x)});
\draw[dashed, smooth, samples = 20, domain = -4.5:-3.5] plot(\x, {0.3*exp(2*(-3-\x))});
\draw[dashed, smooth, samples = 20, domain = -7.5:-6.5] plot(\x, {0.3*exp(2*(-6-\x))});
\draw[dashed, smooth, samples = 20, domain = -10.5:-9.5] plot(\x, {0.3*exp(2*(-9-\x))});

\draw (-0.5, 0.87) node[below]{$\Gamma\left(\lambda-1,\dfrac{2^{k-1}\beta_2}{|c_j|}\right)$};
\draw (-4.5, 6) node[above]{$\Gamma\left(\lambda-1,\dfrac{2^{k}\beta_2}{|c_j|}\right)$};
\draw (-6.5, 0.87) node[below]{$\Gamma\left(\lambda-1,\dfrac{2^{k+1}\beta_2}{|c_j|}\right)$};
\draw (-10.5, 6) node[above]{$\Gamma\left(\lambda-1,\dfrac{2^{k+2}\beta_2}{|c_j|}\right)$};

\draw[dashed] (-11, 5.15) -- (0,5.15);
\draw (0, 5.15) node[right]{$\im w = t_{n,k}^j$};

\draw[ultra thick, smooth, samples = 10, domain = -4.35:-4.26] plot(\x, {0.3*exp(2*(-\x-3))});
\draw[ultra thick, smooth, samples = 10, domain = -7.35:-7.26] plot(\x, {0.3*exp(2*(-\x-6))});
\draw[ultra thick] (-7.26, 3.75) -- (-4.26, 3.75);
\draw[ultra thick] (-7.35, 4.45) -- (-4.35, 4.45);

\draw[ultra thick, smooth, samples = 10, domain = -1.4:-0.92] plot(\x, {0.3*exp(-2*\x)});
\draw[ultra thick, smooth, samples = 10, domain = -10.4:-9.92] plot(\x, {0.3*exp(2*(-\x-9))});
\draw[ultra thick] (-9.92, 1.9) -- (-0.92, 1.9);
\draw[ultra thick] (-10.4, 4.9) -- (-1.4, 4.9);

\draw (-8.61, 3.1) node{$\tilde \cT_{n,k}^j$};
\draw (-5.75 , 4.1) node{$\cT_{n,k}^j$};

\end{tikzpicture}
\caption{An illustration of the sets $\cT_{n,k}^j$ and $\tilde \cT_{n,k}^j$.}  
\label{fig_T}
\end{figure}

Moreover, let $\cU_{n,k}^j$ be the rectangle containing all $v\in\dC$ satisfying
$$\frac{11}{16}\left(\frac{2^{k}\beta_2}{d}\right)^dr_{n,k}^j<\re v<\frac{5}{4}\left(\frac{2^{k+1}\beta_2}{d}\right)^dr_{n,k}^j$$
and
$$\frac{3}{8\pi}\left(\frac{2^k\beta_2}{d}\right)^dr_{n,k}^jd\theta_2<\im v<\frac{5}{2}\left(\frac{2^{k+1}\beta_2}{d}\right)^dr_{n,k}^jd\theta_2.$$
Also, let $\tilde \cU_{n,k}^j$ be the rectangle containing all $v\in\dC$ satisfying
$$\frac{5}{8}\left(\frac{2^{k}\beta_2}{d}\right)^dr_{n,k}^j<\re v<\frac{11}{8}\left(\frac{2^{k+1}\beta_2}{d}\right)^dr_{n,k}^j$$
and
$$\frac{1}{4\pi}\left(\frac{2^{k}\beta_2}{d}\right)^dr_{n,k}^jd\theta_2<\im v<3\left(\frac{2^{k+1}\beta_2}{d}\right)^dr_{n,k}^jd\theta_2.$$
Note that $\cU_{n,k}^j\subset\tilde\cU_{n,k}^j$.

\begin{lemma}   \label{lemma_UinH}
There is $n_0\in\dN$ such that for all $n\in\dZ$ with $|n|\ge n_0$, $k\in\dN$ and $j\in\{1,...,d\}$, we have
$$\tilde \cU_{n,k}^j\subset \cH\left(\lambda,\frac{1}{c^*},\nu\right)$$
with $c^*=\max_l|c_l|$ as defined in \eqref{eq_cstar}.
\end{lemma}

\begin{proof}
Let $v\in \tilde \cU_{n,k}^j$. Note that $r_{n,k}^j\to\infty$ as $|n|\to\infty$ uniformly in $k$. In particular,
$$\im v>\frac{1}{4\pi}\left(\frac{2^{k}\beta_2}{d}\right)^dr_{n,k}^jd\theta_2\ge\nu$$
if $|n|$ is sufficiently large. Also, 
$$\im v<3\left(\frac{2^{k+1}\beta_2}{d}\right)^dr_{n,k}^jd\theta_2=\frac{24}{5}2^dd\theta_2\cdot\frac{5}{8}\left(\frac{2^k\beta_2}{d}\right)^dr_{n,k}^j<\frac{24}{5}2^dd\theta_2\re v,$$
and hence
$$|v|\le|\re v|+|\im v|<\left(1+\frac{24}{5}2^dd\theta_2\right)\re v.$$
Thus,
$$\re v\ge\frac{1}{1+(24/5)2^dd\theta_2}|v|\ge\lambda\log|v|-\log\frac{1}{c^*}$$
if $|n|$ and hence $r_{n,k}^j$ and $|v|$ are sufficiently large. 
\end{proof}

\begin{lemma}   \label{lemma_image_T}
There are $\beta_2, \nu>0$ such that for all $j\in\{1,...,d\}$, $k\in\dN$ and $n\in\dZ$ with $|t_{n,k}^j|>\nu+4^{d+1}\pi\theta_2$, we have
$$h_j(\cT_{n,k}^j)\subset \cU_{n,k}^j\quad\text{and } h_j(\tilde \cT_{n,k}^j)\supset \tilde \cU_{n,k}^j.$$
\end{lemma}

\begin{proof}
First suppose that $w\in\cT_{n,k}^j$. Then by Lemma \ref{lemma_map_links} and the fact that 
$\theta_2<1/(2\cdot4^{d+1}d\pi)\arccos(11/12)<1/(2d)\arccos(11/12),$ we have
$$\re h_j(w)>\frac{3}{4}\left(\frac{2^k\beta_2}{d}\right)^dr_{n,k}^j\cos(2d\theta_2)>\frac{11}{16}\left(\frac{2^k\beta_2}{d}\right)^dr_{n,k}^j,$$
$$\re h_j(w)<\frac{5}{4}\left(\frac{2^{k+1}\beta_2}{d}\right)^dr_{n,k}^j,$$
$$\im h_j(w)>\frac{3}{8\pi}\left(\frac{2^k\beta_2}{d}\right)^dr_{n,k}^jd\theta_2,$$
$$\im h_j(w)<\frac{5}{2}\left(\frac{2^{k+1}\beta_2}{d}\right)^dr_{n,k}^jd\theta_2.$$
Hence, $h_j(\cT_{n,k}^j)\subset\cU_{n,k}^j$.

Also, Lemma \ref{lemma_map_links} yields the following. If $w\in\Gamma(\lambda-1,2^{k-1}\beta_2/|c_j|)$ with $t_{n,k}^j-4^{d+1}\pi\theta_2\le\im w\le t_{n,k}^j-1/(10\cdot4^d\pi)\theta_2$, then
$$\re h_j(w)<\frac{5}{4}\left(\frac{2^{k-1}\beta_2}{d}\right)^dr_{n,k}^j\le\frac{5}{8}\left(\frac{2^k\beta_2}{d}\right)^dr_{n,k}^j.$$
If $w\in\Gamma(\lambda-1,2^{k+2}\beta_2/|c_j|)$ with $t_{n,k}^j-4^{d+1}\pi\theta_2\le\im w\le t_{n,k}^j-1/(10\cdot4^d\pi)\theta_2$, then using that $\theta_2<1/(2\cdot4^{d+1}d\pi)\arccos(11/12)$, we get
$$\re h_j(w)>\frac{3}{4}\left(\frac{2^{k+2}\beta_2}{d}\right)^dr_{n,k}^j\cos(2d4^{d+1}\pi\theta_2)>\frac{11}{8}\left(\frac{2^{k+1}\beta_2}{d}\right)^dr_{n,k}^j.$$
If $w\in \cH(\lambda-1,2^{k+2}\beta_2/|c_j|,\nu)\setminus \cH(\lambda-1,2^{k-1}\beta_2/|c_j|,\nu)$ and $\im w = t_{n,k}^j-4^{d+1}\pi\theta_2$, then
$$\im h_j(w)>3\left(\frac{2^{k+1}\beta_2}{d}\right)^dr_{n,k}^jd \theta_2.$$
If $w\in \cH(\lambda-1,2^{k+2}\beta_2/|c_j|,\nu)\setminus \cH(\lambda-1,2^{k-1}\beta_2/|c_j|,\nu)$ and $\im w = t_{n,k}^j-1/(10\cdot4^d\pi)\theta_2$, then
$$\im h_j(w)<\frac{1}{4\pi}\left(\frac{2^{k}\beta_2}{d}\right)^dr_{n,k}^jd\theta_2.$$
Thus, $h_j(\partial \tilde \cT_{n,k}^j)\cap \tilde \cU_{n,k}^j=\emptyset$. Since $\cT_{n,k}^j\subset\tilde \cT_{n,k}^j$ and $h_j(\cT_{n,k}^j)\subset \cU_{n,k}^j\subset\tilde \cU_{n,k}^j$, we obtain that $h_j(\tilde \cT_{n,k}^j)\supset \tilde \cU_{n,k}^j.$
\end{proof}

Next, we show that the density of $q(\cF_j)$ in $\tilde \cT_{n,k}^j$ is bounded below by a positive constant.

\begin{lemma}  \label{lemma_FinT}
There are $\delta>0$ and $n_1\in\dN$ such that for all $j\in\{1,...,d\}$, $k\in\dN$ and $n\in\dZ$ with $|n|\ge n_1$, we have
$$\dens(q(\cF_j), \tilde\cT_{n,k}^j)\ge\delta.$$
\end{lemma}

\begin{proof}
We only sketch the proof, since it is similar to the one of Lemma \ref{lemma_FinQ_mitte}.  By Lemma \ref{lemma_UinH},
$$\tilde\cU_{n,k}^j\subset\cH\left(\lambda,\frac{1}{c^*},\nu\right).$$ 
By Lemma \ref{lemma_image_T}, $h_j(\tilde \cT_{n,k}^j)\supset \tilde \cU_{n,k}^j$ and $h_j(\cT_{n,k}^j)\subset \cU_{n,k}^j$. Let $\cV\subset\tilde\cT_{n,k}^j$ be the component of $h_j^{-1}(\tilde\cU_{n,k}^j)$ containing $\cT_{n,k}^j$.
As in the proof of Lemma \ref{lemma_FinQ_mitte}, we get that $f(\vi_j(\cV))\subset \cS_l$ for some $l\in\{1,...,d\}$, and that
$$\dens(q(\cF_j), \tilde \cT_{n,k}^j)\ge c\dens(q(\cF_l), \cU_{n,k}^j)\cdot\dens(\cT_{n,k}^j,\tilde \cT_{n,k}^j)$$
for some $c>0$ independent of $n,k$ and $j$. If $|n|$ and hence $r_{n,k}^j$ is sufficiently large, then by Lemma \ref{lemma_F_rechts}, 
$$\dens(q(\cF_l), \cU_{n,k}^j)\ge\eta_0.$$
Also, the density of $\cT_{n,k}^j$ in $\tilde\cT_{n,k}^j$ is bounded below independent of $n,k$ and $j$, whence the claim follows.
\end{proof}

The final result of this section says that the density of $q(\cF_j)$ in large rectangles in $\cG_\nu\setminus\cH(\lambda-1, \beta_2/|c_j|,\nu)$ is bounded below. 

\begin{lemma}  \label{lemma_F_links}
There are $\beta_2,\nu,D_2,\eta_2>0$ such that for all $j\in\{1,...,d\}$ and any rectangle $\cR\subset \cG_\nu\setminus \cH(\lambda-1,\beta_2/|c_j|,\nu)$ with sides parallel to the real and imaginary axis and side lengths at least $D_2$, we have 
$$\dens(q(\cF_j),\cR)\ge\eta_2.$$
\end{lemma}

\begin{proof}
This is proved the same way as Lemma \ref{lemma_F_mitte}, using Lemma \ref{lemma_FinT}.
\end{proof}


\section{The set \texorpdfstring{$q(\cF(f))$}{q(F(f))}: conclusions}   \label{sec_q(F)_all}
In this section, we combine the results of Sections \ref{sec_q(F)1}-\ref{sec_q(F)3} to show that $q(\cF_j)$ has positive density in large bounded subsets of $\dC$.

\begin{lemma}   \label{lemma_dens_q(F)}
There are $D,\eta_3>0$ such that for all $j\in\{1,...,d\}$ and any square $\cR\subset\dC$ with sides parallel to the real and imaginary axis and side lengths at least $D$, we have
$$\dens(q(\cF_j),\cR)\ge\eta_3.$$
\end{lemma}

\begin{proof}
Let 
$$\cE_1:=\cH\left(\lambda,\frac{\beta_1}{|c_j|},\nu\right)\setminus \cH\left(\lambda,\frac{1}{|c_j|},\nu\right)$$
and
$$\cE_2:=\cH\left(\lambda-1,\frac{\beta_2}{|c_j|},\nu\right)\setminus \cH\left(\lambda-1,\frac{\alpha_1}{|c_j|},\nu\right).$$
Also, let $\gamma_1$ and $\gamma_2$ be the left boundary curves of $\cE_1$ and $\cE_2$, respectively,  parametrised by $y=\im z$. 
Justified by Lemma \ref{lemma_gamma}, we suppose that $\nu$ is so large that 
\begin{equation}  \label{eq_nu}
|\gamma_k'(y)|<\frac{1}{10} \text{ for } |y|\ge\nu \text{ and } k\in\{1,2\}.
\end{equation}
Using the notation of Lemmas \ref{lemma_F_rechts}, \ref{lemma_F_mitte} and \ref{lemma_F_links}, suppose that
\begin{equation}  \label{eq_D-nu-Dk}
D>2\nu+5\max\{D_0,D_1,D_2\}
\end{equation}
and
\begin{equation} \label{eq_D-alpha-beta}
D>20\max\left\{\log\beta_1, \,\log\frac{\beta_2}{\alpha_1}\right\}.
\end{equation}
For $\cS\subset\dC$, let
$$\diam_x(\cS):=\sup\{|\re(z-w)|:\,z,w\in \cS\}$$
and
$$\diam_y(\cS):=\sup\{|\im(z-w)|:\,z,w\in \cS\}.$$
Define
$$\cR_+:=\cR\cap\{z:\,\im z\ge\nu\}, \quad \cR_-:=\cR\cap\{z:\,\im z\le-\nu\},$$
and let 
$$\cR_1:=\begin{cases}\cR_+ &\text{if }\diam_y(\cR_+)\ge\diam_y(\cR_-)\\ \cR_- &\text{otherwise.}\end{cases}$$
By \eqref{eq_D-nu-Dk}, $\diam_y(\cR_1)>\max\{D_0,D_1,D_2\}.$

We now divide $\cR_1$ into $5$ rectangles, $\cR_{1,1},...,\cR_{1,5}$, with $\diam_y(\cR_{1,k})=\diam_y(\cR_1)$ and $\diam_x(\cR_{1,k})=\frac{1}{5}\diam_x(\cR_1)$ for all $k\in\{1,...,5\}$ (see Figure \ref{fig_R}).

\begin{figure}[ht]
\centering
\begin{tikzpicture}
\draw[thick] (0,0) rectangle (5,3);
\foreach \x in {1,...,4} \draw[dashed] (\x,0) --(\x, 3);
\foreach \x in {1,...,5} \draw (\x-0.5, 1.5) node {$\cR_{1,\x}$};
\draw (1, 2.5) node {$\cR_1$};
\end{tikzpicture}
\caption{The rectangle $\cR_1$, bounded by the solid line, is divided into five smaller rectangles by the dashed lines.}
\label{fig_R}
\end{figure}
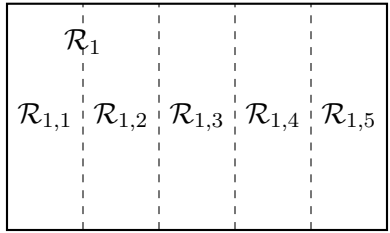

By \eqref{eq_D-nu-Dk}, $\diam_x(\cR_{1,k})>\max\{D_0,D_1,D_2\}.$
By \eqref{eq_nu}, Lemma \ref{lemma_gamma}, \eqref{eq_D-alpha-beta} and the fact that $\cR$ is a square of side length at least $D$, we have
\begin{equation}
\begin{split}
\diam_x (\cE_l\cap \cR)&<\frac{1}{10}\diam_y(\cR)+2\max\left\{\log\beta_1,\log\frac{\beta_2}{\alpha_1}\right\}\\
&<\frac{1}{10}\diam_y(\cR)+\frac{1}{10}D\le\frac{1}{5}\diam_x(\cR)
\end{split}
\end{equation}
for $l\in\{1,2\}$. Thus, $\cE_1$ and $\cE_2$ each intersect at most two of the rectangles $\cR_{1,k}$. Hence, there exists $l\in\{1,...,5\}$ such that $\cR_{1,l}$ does not intersect $\cE_1\cup \cE_2$. This implies
that $\cR_{1,l}$ satisfies the hypothesis of one of Lemmas \ref{lemma_F_rechts}, \ref{lemma_F_mitte} and \ref{lemma_F_links}. Hence, 
$$\dens(q(\cF_j), \cR_{1,l})\ge\min\{\eta_0,\eta_1,\eta_2\}$$
and
\begin{equation}
\begin{split}
\dens(q(\cF_j), \cR)&\ge\dens(q(\cF_j),\cR_{1,l})\cdot\dens(\cR_{1,l},\cR)\\
&\ge\min\{\eta_0,\eta_1,\eta_2\}\cdot\frac{1}{10}\frac{\diam_x(\cR)(\diam_y(\cR)-2\nu)}{(\diam_x\cR)^2}\\
&\ge\min\{\eta_0,\eta_1,\eta_2\}\cdot\frac{1}{10}\left(1-\frac{2\nu}{D}\right). \qedhere
\end{split}
\end{equation}
\end{proof}

The following corollary is an immediate consequence of Lemma \ref{lemma_dens_q(F)}.

\begin{cor}   \label{cor_dens_q(F)}
There are $r_0,\eta>0$ such that for all $z\in\dC$, all $r\ge r_0$ and all $j\in\{1,...,d\}$, we have
$$\dens(q(\cF_j),\cD(z,r))\ge\eta.$$
\end{cor}

\begin{remark}
Corollary \ref{cor_dens_q(F)} says that $\dC\setminus q(\cF_j)$ is thin at $\infty$.  
\end{remark}


\section{Proof of Theorem \ref{thm_main}}   \label{sec_proof}

\begin{proof}[Proof of Theorem \ref{thm_main}]
We will verify the assumptions of Theorem \ref{thm_zero_measure}. By Lemma \ref{lemma_sing_assumptions}, the set $\cP(f)\cap\cJ(f)$ is finite, so it remains to prove that there exists $R_1>0$ such that $\cJ(f)$ is uniformly thin at $(\cP(f)\cap\dC)\setminus\overline{\cD(0,R_1)}$ and that $\cJ(f)$ is thin at $\infty$. Let $r_1>0$ such that
\begin{enumerate}[(a)]
\item $|q'(z)|\ge(d/2)|z|^{d-1}$ for all $z\in\dC$ with $|z|\ge r_1$;
\item  each $z_0\in \cP(f)$ with $|z_0|\ge r_1$ is a  zero of $g$ and hence a superattracting fixed point of $f$. Justified by Corollary \ref{cor_zeros}, we also assume that there is $j\in\{1,...,d\}$ with $z_0\in\cS_j$, and $\dist(z_0,\partial \cS_j)\ge3$. Moreover, suppose that the conclusion of Lemma \ref{lemma_disk_attr_basin} holds for $|z_0|\ge r_1$.
\end{enumerate}
Let $r_0$ be the constant from Corollary \ref{cor_dens_q(F)}. First, we will show that there exists $\eta_4>0$ such that for all $j\in\{1,...,d\}$, all $z\in \cS_j$ with $|z|\ge r_1$ and all $r>8r_0/(d|z|^{d-1})$ with $\cD(z,2r)\subset \cS_j$, we have
\begin{equation}  \label{eq_dens_first}
\dens(\cF(f), \cD(z,r))\ge\eta_4.
\end{equation}

Recall that $q$ is injective in $\cS_j$. By Koebe's theorems,
$$\cD\left(q(z),\frac{1}{4}|q'(z)|r\right)\subset q(\cD(z,r))\subset \cD(q(z),4|q'(z)|r).$$
By (a) and the assumption on $r$, we have $(1/4)|q'(z)|r\ge r_0$. Hence, by Corollary \ref{cor_dens_q(F)},
$$\dens\left(q(\cF_j),\cD\left(q(z),\frac{1}{4}|q'(z)|r\right)\right)\ge\eta.$$
Thus,
\begin{equation}
\begin{split}
&\dens(q(\cF_j), q(\cD(z,r)))\\
&\ge\dens\left(\cD\left(q(z),\frac{1}{4}|q'(z)|r\right),q(\cD(z,r))\right)\cdot\dens\left(q(\cF_j),\cD\left(q(z),\frac{1}{4}|q'(z)|r\right)\right)\\
&\ge\dens\left(\cD\left(q(z),\frac{1}{4}|q'(z)|r\right),\cD(q(z),4|q'(z)|r)\right)\cdot\eta
=\frac{1}{256}\eta.
\end{split}
\end{equation}
By the Koebe distortion theorem,
$$\dens(\cF(f),\cD(z,r))\ge\left(\frac{\min_{\zeta\in \cD(z,r)}|q'(\zeta)|}{\max_{\zeta\in \cD(z,r)}|q'(\zeta)|}\right)^2\dens(q(\cF_j),q(\cD(z,r)))\ge\frac{1}{3^8\cdot256}\eta.$$
This implies \eqref{eq_dens_first} with $\eta_4=\eta/(3^8\cdot256).$

Let us now prove that there exists $R_1>0$ such that $\cJ(f)$ is uniformly thin at $(\cP(f)\cap\dC)\setminus\overline{\cD(0,R_1)}$. Let $\delta_1\in(0,1)$, $z_0\in\cP(f)$ with $|z_0|> r_1+1$ and $z\in \cD(z_0,\delta_1).$ By (b), $\cD(z,2\delta_1)\subset \cS_j$. Also, $|z|\ge r_1$. If $|z-z_0|\ge8r_0/(d|z|^{d-1})$, then by \eqref{eq_dens_first},
$$\dens(\cF(f),\cD(z,|z-z_0|))\ge\eta_4.$$
Now suppose that 
\begin{equation} \label{eq_z_closeto_z0}
|z-z_0|<\dfrac{8r_0}{d|z|^{d-1}}. 
\end{equation}
By Lemma \ref{lemma_disk_attr_basin},  we have $\cD(z_0,1/(3d|z_0|^{d-1}))\subset \cF(f)$. Hence, 
$$\dens(\cF(f),\cD(z,|z-z_0|))\ge\dens\left(\cD\left(z_0,\dfrac{1}{3d|z_0|^{d-1}}\right),\cD(z,|z-z_0|)\right).$$
The expression on the right hand side is bounded below independent of $z_0$ and $|z|$, provided \eqref{eq_z_closeto_z0} is satisfied. So $\cJ(f)$ is uniformly thin at $(\cP(f)\cap\dC)\setminus\overline{\cD(0,r_1+1)}$.

It  remains to prove that $\cJ(f)$ is thin at $\infty$. Let $R$ be as in Section \ref{sec_changeOfVar} and let $r_2>\max\{2R^{1/d}, r_1\}$. If $r_2$ is sufficiently large, then Lemma \ref{lemma_change_of_var} yields that $\bigcup_{j=1}^d\partial\cS_j\setminus\cD(0,r_2)$ is contained in $d$ pairwise disjoint halfstrips, $\cT_1,...,\cT_d$, of width $1$. We can assume that $r_2$ is so large that $\dist(\cT_k,\cT_l)\ge 1$ for $k\ne l$. Then for $|z|\ge r_2+3$, the set $\cD(z,3)\setminus\bigcup_{j=1}^d\cT_j$ contains a disk, $\cD$, of radius $1/2$. There is $j\in\{1,...,d\}$ with $\cD\subset\cS_j$. Let $\cD'$ be the disk with the same center as $\cD$ and radius $1/4$. If $r_2$ is sufficiently large, then by \eqref{eq_dens_first}, we have $\dens(\cF(f),\cD')\ge\eta_4$, and hence
$$\dens(\cF(f),\cD(z,3))\ge\dens(\cF(f),\cD')\cdot\dens(\cD', \cD(z,3))\ge\frac{\eta_4}{144}.$$

We now consider the case that $|z|<r_2+3$. Let $\zeta_1,...,\zeta_n\in\cD(0,r_2+3)$ such that 
$$\cD(0,r_2+3)\subset\bigcup_{k=1}^n\cD(\zeta_k,1).$$
Then
$$\eta_5:=\min_{1\le k\le n}\dens(\cF(f), \cD(\zeta_k,1))>0.$$
For $z\in \cD(0,r_2+3)$, let $k\in\{1,...,n\}$ such that $z\in \cD(\zeta_k,1)$. Then $\cD(\zeta_k,1)\subset \cD(z,3)$ and
$$\dens(\cF(f),\cD(z,3))\ge\frac{1}{9}\dens(\cF(f),\cD(\zeta_k,1))\ge\frac{1}{9}\eta_5.$$
Thus, $\cJ(f)$ is thin at $\infty$. Hence, Theorem \ref{thm_zero_measure} yields that $\cJ(f)$ has Lebesgue measure zero.
\end{proof}

\paragraph{Acknowledgements}
I would like to thank Walter Bergweiler for helpful suggestions and for carefully reading the manuscript.
\bibliographystyle{plain}

\bibliography{Julia_sets_of_zero_measure}

\end{document}